\documentclass[12pt,a4paper]{amsart}
 
\usepackage{amsmath,amssymb,amsthm,bbm} 
\usepackage{enumitem}  
\usepackage[overload]{textcase}
 
\usepackage{epsfig,psfrag,color} 

\setlist{itemsep=1pt,parsep=0pt,topsep=2pt,partopsep=0pt}  
\setenumerate{leftmargin=*} 
 
\def\itm#1{\rm ({#1})} 
\def\itmit#1{\itm{\it #1\,}}

\def\irom{\itmit{\roman{*}}} 
\def\iabc{\itmit{\alph{*}}} 
 
\def\endofFact{\hfill\scalebox{.6}{$\Box$}}

\let\subset\subseteq  
\let\eps\varepsilon 
\let\rho\varrho 
\def\dcup{\dot\cup}  
\def\bound{\partial} 
 
\def\subsc#1{\textsc{\MakeTextLowercase{#1}}} 

\newtheorem{theorem}{Theorem}
\newtheorem{lemma}[theorem] {Lemma}    
    
\newtheorem{conjecture}[theorem] {Conjecture}   
\newtheorem{proposition}[theorem] {Proposition}   
  
\newtheorem{fact}[theorem]{Fact}

\theoremstyle{remark}

\newcommand{\By}[2]{\overset{\mbox{\tiny{#1}}}{#2}} 
\newcommand{\ByRef}[2]{   \By{\eqref{#1}}{#2} }

\newcommand{\gBy}[1]{     \By{#1}{>} } 
\newcommand{\leBy}[1]{    \By{#1}{\le} } 
\newcommand{\geBy}[1]{    \By{#1}{\ge} } 
\newcommand{\eqByRef}[1]{ \ByRef{#1}{=} } 
 
\newcommand{\gByRef}[1]{  \ByRef{#1}{>} } 
\newcommand{\leByRef}[1]{ \ByRef{#1}{\le} } 
\newcommand{\geByRef}[1]{ \ByRef{#1}{\ge} }

\newcommand{\NATS}{\mathbb{N}}

\DeclareMathOperator{\intr}{int} 
\DeclareMathOperator{\CTF}{CTF} 
\DeclareMathOperator{\sqp}{sp} 
\DeclareMathOperator{\sqc}{sc} 
\newcommand{\rp}{r_p} 
\newcommand{\rc}{r_c} 

\newcommand{\overrightharp}[1]{{\overset{\scalebox{1.8}[.7]{\!\!\!\raisebox{-.5pt}{$\rightharpoonup$}}}{\smash{#1}\vphantom{a}}}}

\newcommand{\EMAIL}[1]{  \textit{E-mail}: \texttt{#1} } 

\newcommand{\oldqed}{}
\newenvironment{factproof}[1][Proof]{
  \renewcommand{\oldqed}{\qedsymbol}
  \renewcommand{\qedsymbol}{\endofFact}
  \begin{proof}[#1]
}{
  \end{proof}
  \renewcommand{\qedsymbol}{\oldqed}
}

\title[Between Tur\'an's theorem and P\'osa's conjecture]{Filling the gap
  between Tur\'an's theorem and P\'osa's conjecture}

  \author[Peter Allen]{Peter Allen*}
  \thanks{
    *
    DIMAP and Mathematics Institute, University of
    Warwick, Coventry, CV4 7AL, United Kingdom.
    \EMAIL{P.D.Allen@warwick.ac.uk}
  }

  \author[Julia B\"ottcher]{Julia B\"ottcher\dag}
  \thanks{
    \dag
    Zentrum Mathematik, Technische Universit\"at M\"unchen,
    Boltzmannstra\ss{}e~3, D-85747 Garching bei
    M\"unchen, Germany. 
    \EMAIL{boettche@ma.tum.de}
  }

  \author[Jan Hladk\'y]{Jan Hladk\'y\ddag}
  \thanks{
    \ddag
    Department of Applied Mathematics, Faculty of Mathematics and Physics, Charles University, Malostransk\'e n\'am\v est\'i
    25, 118~00, Prague, Czech Republic and
    DIMAP and Department of Computer Science, University of
    Warwick, Coventry, CV4 7AL, United Kingdom.
    \EMAIL{honzahladky@gmail.com}
  }

  \thanks{PA was partially supported by DIMAP, EPSRC award EP/D063191/1, 
    JB by DFG grant TA 309/2-1,
    JH by the Charles University grant GAUK 202-10/258009, by
    DAAD, by BAYHOST, and by DIMAP, EPSRC award EP/D063191/1.
  } 
 
\date{\today} 

\begin{document} 
 
\begin{abstract} 
  Much of extremal graph theory has concentrated either on finding very small
  subgraphs of a large graph (Tur\'an-type results) or on finding spanning
  subgraphs (Dirac-type results).  In this paper we are interested in finding
  intermediate-sized subgraphs.
  We investigate minimum degree conditions under which a
  graph~$G$ contains squared paths and squared cycles of arbitrary specified lengths.  
  We determine precise thresholds, assuming that the order of~$G$ is large. 
  This extends results of Fan and Kierstead [J.\ Combin.\ Theory Ser.\ B 63 
  (1995), 55--64] and of Koml\'os, Sark\"ozy, and Szemer\'edi [Random 
  Structures Algorithms 9 (1996), 193--211] concerning the containment of a 
  spanning squared path and a spanning squared cycle, respectively. 
  Our results show that such minimum degree conditions constitute not merely
  an interpolation between the corresponding Tur\'an-type and Dirac-type
  results, but exhibit other interesting phenomena.
\end{abstract}  

\maketitle

\section{Introduction} \label{sec:Intro} 
One of the main programmes of extremal graph theory is the study of conditions 
on the vertex degrees of a host graph~$G$ under which a target  graph~$H$ 
appears as a subgraph of~$G$ (which we denote by~$H\subset G$). Tur\'an's 
theorem~\cite{Tur} is a prominent example for results of this type. It asserts 
that an average degree $d(G)>\frac{r-2}{r-1}n$ forces the copy of a complete 
graph~$K_r$ in~$G$ (and that this is best possible), 
where here and throughout~$n$ is the number of vertices in the host graph~$G$. 
%
%
More generally, the celebrated theorem of Erd\H os and Stone~\cite{ErdSto}
implies that for a \emph{fixed} graph~$H$ the chromatic number~$\chi(H)$ of~$H$
determines the average degree that is necessary to guarantee a copy of~$H$: If
$H$ has chromatic number $\chi(H)=r$ and $d(G)\ge(\frac{r-2}{r-1}+o(1))n$,
then~$H$ is a subgraph of~$G$.  This settles the problem for fixed target graphs
(with chromatic number at least~$3$), that is, graphs that are `small' compared
to the host graph.
 
 
\smallskip
 
%
Dirac's theorem~\cite{Dir}, another classical result from the area, considers 
target graphs that are of the same order as the host graph, i.e., so-called 
\emph{spanning} target graphs.  Clearly, any average degree condition on the 
host graph that enforces a connected spanning subgraph must be trivial, and 
hence the average degree needs a suitable replacement in this setting. 
Here, the minimum degree is a natural candidate, and indeed, Dirac's theorem 
asserts that every graph~$G$ with minimum degree~$\delta(G)>\frac12 n$ has a 
Hamilton cycle. This implies in particular that~$G$ has a matching covering 
$2\lfloor n/2\rfloor$ vertices.  
 
A $3$-chromatic version of this matching result follows from a theorem by Corr\'adi
and Hajnal~\cite{CorHaj}: the minimum degree condition
$\delta(G)\ge2\lfloor n/3 \rfloor$ implies the existence of a so-called
\emph{spanning triangle factor} in~$G$, that is, a collection of $\lfloor
n/3\rfloor$ vertex disjoint triangles. A well-known conjecture of P\'osa
(see, e.g., \cite{ErdPos}) asserts that roughly the same minimum degree
actually guarantees the existence of a connected super-graph of a spanning
triangle factor. It states that any graph~$G$ with $\delta(G)\ge \frac23 n$
contains a spanning \emph{squared cycle}~$C^2_n$, where the square of a graph,~$F^2$, is obtained
from~$F$ by adding edges between all pairs of vertices with
distance~$2$ in $F$. This can be seen as a $3$-chromatic analogue of Dirac's
theorem, which turned out to be much more difficult than its $2$-chromatic
cousin.
%
 
Fan and Kierstead~\cite{FanKie95} proved an approximate version of P\'osa's 
conjecture for large~$n$. In addition they determined a sufficient and best 
possible minimum degree condition for the case that the squared cycle in 
P\'osa's conjecture is replaced by a \emph{squared path} $P^2_n$, i.e., the 
square of a spanning path~$P_n$. 
 
\begin{theorem}[Fan \& Kierstead~\cite{FanKie96}]\label{thm:FanKierstead} 
  If $G$ is a  graph on $n$ vertices with minimum degree $\delta(G)\ge (2n-1)/3$, 
  then $G$ contains a spanning squared path $P_n^2$. 
\end{theorem} 
 
The P\'osa Conjecture was verified for large values of $n$ by Koml\'os, 
Sark\"ozy, and Szemer\'edi~\cite{KSS_pos}. The proof in~\cite{KSS_pos} actually 
asserts the following stronger result, which guarantees not only spanning 
squared cycles but additionally squared cycles of all lengths between~$3$ 
and~$n$ that are divisible by~$3$. 
 
\begin{theorem}[Koml\'os, S\'ark\"ozy \& Szemer\'edi~\cite{KSS_pos}] 
\label{thm:KSSPosa} 
  There exists an integer $n_0$ such that for all integers $n>n_0$  
  any graph $G$ of order $n$ and minimum degree $\delta(G)\ge\frac23n$ 
  contains all squared cycles $C^2_{3\ell}\subset G$ with $3\le 3\ell\le n$.  
  If furthermore $K_4\subset G$, then $C^2_{\ell}\subset G$ for any 
  $3\le\ell\le n$ with $\ell\neq 5$. 
\end{theorem} 
 
 
For squared cycles $C^2_{\ell}$ with~$\ell$ not divisible by~$3$ the additional 
condition $K_4\subset G$ is necessary because these target graphs are not 
$3$-colourable and hence a complete $3$-partite graph shows that one cannot 
hope to force~$C^2_{\ell}$ unless $\delta(G)\ge (2n+1)/3$. If $\delta(G)\ge 
(2n+1)/3$, on the other hand, then Tur\'an's Theorem asserts that~$G$ contains 
a copy of~$K_4$ and hence Theorem~\ref{thm:KSSPosa} implies $C^2_{\ell}\subset 
G$ for any $3\le\ell\le n$ with $l\neq 5$. The case $\ell=5$ has to be 
excluded because $C^2_5$ is the $5$-chromatic~$K_5$. 
 
 
\smallskip
 
In this paper we address the question of what happens between these two extrema 
of target graphs with constant order and order~$n$. We are 
interested in essentially best possible minimum degree conditions that enforce subgraphs 
covering a certain percentage of the host graph. 
 
Let us start with a simple example. It is easy to see that every graph~$G$ with 
minimum degree $\delta(G)\ge\delta$ for $0\le\delta\le\frac12n$ has a matching 
covering at least $2\delta$ vertices (see
Proposition~\ref{prop:match}\ref{prop:match:a}). This gives a linear dependence
between the forced size of a matching in the host graph and its minimum
degree. A more general form of the result of Corr\'adi and Hajnal~\cite{CorHaj} 
mentioned earlier is a variant of this linear dependence for triangle factors. 
 
 
\begin{theorem}[Corr\'adi \& Hajnal~\cite{CorHaj}] 
\label{thm:CorHaj} 
  Let $G$ be a graph on $n$ vertices with minimum degree 
  $\delta(G)=\delta\in[\frac12n,\frac23n]$. Then $G$ contains $2\delta-n$ vertex 
  disjoint triangles. 
\end{theorem} 
 
The main theorem of this paper is a  corresponding result mediating between 
Tur\'an's theorem and P\'osa's conjecture. More precisely, our aim is 
to provide exact minimum degree thresholds for the 
appearance of a squared path~$P_\ell^2$ and a squared cycle~$C_{\ell}^2$. 
 
 
There are at least two reasonable guesses one might make as 
to what minimum degree~$\delta(G)=\delta$ will guarantee which 
length~$\ell=\ell(n,\delta)$ of squared path (or longest squared cycle). On the
one hand, the degree threshold for a \emph{spanning} squared path or cycle and for a 
spanning triangle factor are approximately the same. So perhaps this remains 
true for smaller~$\ell$: in light of Theorem~\ref{thm:CorHaj} one could expect 
that $\ell(n,\delta)$ were roughly $3(2\delta(G)-n)$. This turns out to be far 
too optimistic. 
 
On the other hand, proofs of preceding results dealing with spanning subgraphs 
essentially combine greedy techniques with local changes. They simply start to 
construct the desired subgraph in (almost) any location, and in the event of 
getting stuck change only a few of the vertices embedded so far; at no time do 
they scrap an entire half-constructed object and start anew. It would not be 
unreasonable to believe that this technique also leads to best possible 
minimum degree conditions for large but not spanning subgraphs.  
Clearly, in the case of (unsquared) paths such a greedy strategy  
provides a path of length~$\delta(G)+1$. As~$G$ might be disconnected, however, 
it cannot guarantee longer paths if $\delta(G)<n/2$.  
For squared paths the following construction shows that with an arbitrary 
starting location one cannot hope for squared paths on more than 
$\frac32(2\delta(G)-n)$ vertices: If~$G$ contains disjoint cliques~$C$ and~$C'$ of orders 
$2\delta-n$ and $n-\delta$, and an independent set~$I$ of order $n-\delta$ such that all vertices 
of~$C$ and $C'$ are connected to all vertices of~$I$ but not to other vertices of~$G$, then it 
is not difficult to see that the longest squared path in~$G$ starting in an edge 
of~$C$ has length~$\frac32(2\delta(G)-n)$. This could lead to the idea 
that~$\ell(n,\delta)$ were approximately $\frac32(2\delta(G)-n)$. It is 
true that there are squared paths of this length in~$G$---but this lower bound
is almost always excessively pessimistic. In other words, it turns out that
one has to carefully choose the `region' of~$G$ to look for the desired
squared path. Since spanning squared paths use all vertices of~$G$ this
problem does not occur for these subgraphs.
 
For fixed~$n$ both guesses propose a linear dependence between~$\delta$ and the 
length~$\ell(n,\delta)$ of a forced squared path (or cycle). As we will see below $\ell(n,\delta)$ 
as a function of $\delta$ behaves very differently: it is piece-wise linear 
but jumps at certain points. (These jumps can be viewed as phase transitions for the
appearance of squared paths or cycles.) To make this precise we introduce the following 
functions. Given two positive integers $n$ and $\delta$ with 
$\delta\in(\frac12n,n-1]$, we define $\rp(n,\delta)$ to be
the largest integer
$r$ such  that $n-\delta+\lfloor\delta/r\rfloor>\delta$
and $\rc(n,\delta)$ to be the largest integer $r$ such that $n-\delta+\lceil\delta/r\rceil>\delta$. We then 
define 
\begin{equation}\label{eq:sqpc}
\begin{split}
   \sqp(n,\delta) & := 
   \min\Big\{\, 
     \Big\lceil  
       \tfrac32\lceil\delta/\rp(n,\delta)\rceil+\tfrac12 
     \Big\rceil,\, 
     n\, 
   \Big\} \,, 
   \quad\text{and} \\
   \sqc(n,\delta) & := 
   \min\Big\{\, 
     \Big\lfloor 
       \tfrac32\lceil\delta/\rc(n,\delta)\rceil 
     \,\Big\rfloor,\, 
     n\, 
   \Big\}\,. 
\end{split}
\end{equation} 
Observe that $\sqc(n,\delta)\le \sqp(n,\delta)$ and that for almost every
$\alpha\in(0,1)$ we have
$\lim_{n\to\infty}\sqc(n,\alpha n)/n=\lim_{n\to\infty}\sqp(n,\alpha n)/n$. The
dependence between $\sqp(n,\delta)$ and~$\delta$ is illustrated in Figure~\ref{fig:sqpplot}. 
 
\begin{figure}[htbp] 
\begin{center} 
    \psfrag{gamma}{\hspace{7mm}$\delta$} 
    \psfrag{3gamma-1.5}{\hspace{-3mm}\scalebox{.9}{$\frac32(2\delta-n)$}} 
    \psfrag{4gamma-2}{\hspace{0mm}\scalebox{.9}{\ $4\delta-2n$}} 
    \psfrag{6gamma-3}{\hspace{0mm}\scalebox{.9}{\ $6\delta-3n$}} 
    \psfrag{sp(gamma)}{\hspace{0mm}\scalebox{.9}{$\sqp(n,\delta)$}} 
    \psfrag{ 1.2}{} 
    \psfrag{sp}{} 
    \psfrag{ 1}{\scalebox{.9}{$n$}} 
    \psfrag{ 0.8}{\scalebox{.9}{$\frac{4n}{5}$}} 
    \psfrag{ 0.6}{\scalebox{.9}{$\frac{3n}{5}$}} 
    \psfrag{ 0.4}{\scalebox{.9}{$\frac{2n}{5}$}} 
    \psfrag{ 0.2}{\scalebox{.9}{$\frac{n}{5}$}} 
    \psfrag{ 0}{\scalebox{.9}{$0$}} 
    \psfrag{ 0.5}{\scalebox{.9}{$\frac{n}{2}$}} 
    \psfrag{ 6/11}{\scalebox{.9}{$\frac{6n}{11}$}} 
    \psfrag{ 5/9}{\scalebox{.9}{$\frac{5n}{9}$}} 
    \psfrag{ 4/7}{\scalebox{.9}{$\frac{4n}{7}$}} 
    \psfrag{ 2/3}{\scalebox{.9}{$\frac{2n}{3}$}} 
   \includegraphics[height=8cm,width=12.5cm]{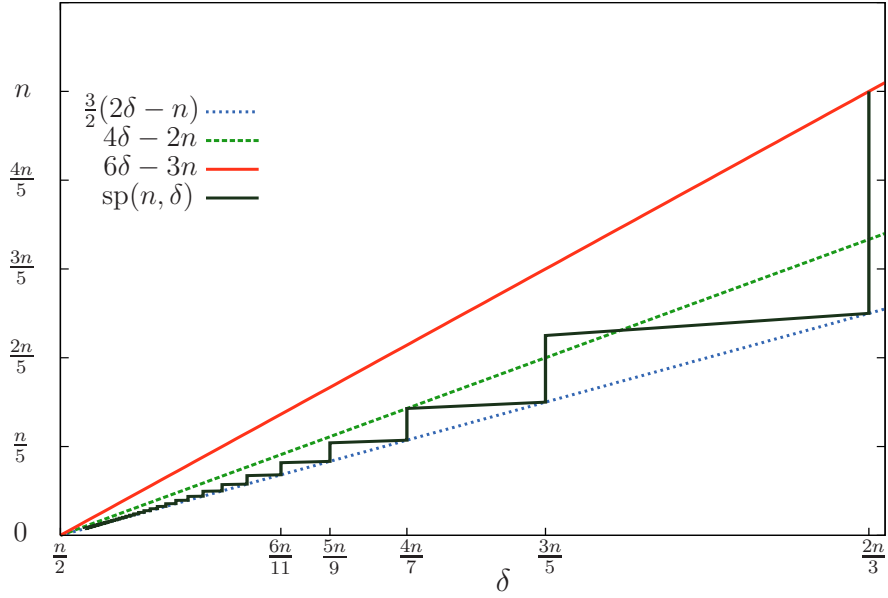} 
  \end{center} 
  \caption{The behaviour of $\sqp(n,\delta)$.} 
\label{fig:sqpplot} 
\end{figure} 
 
Our main theorem now states states that $\sqp(n,\delta)$ and $\sqc(n,\delta)$
are the maximal lengths of squared paths and cycles, respectively, forced in an
$n$-vertex graph~$G$ with minimum degree~$\delta$. More generally, and in
accordance with Theorem~\ref{thm:KSSPosa}, we show that~$G$ also contains any
shorter squared cycle with length divisible by~$3$ (see~\ref{thm:main:i} of
Theorem~\ref{thm:main}). We shall show below that these results are tight by
explicitly constructing extremal graphs~$G_p(n,\delta)$ and~$G_c(n,\delta)$ for
squared paths and cycles.  While the extremal graphs of all previously discussed
results are Tur\'an graphs (complete $r$-partite graphs, where~$r=3$ in the case
of squared paths and cycles) the graphs~$G_p(n,\delta)$ and~$G_c(n,\delta)$ have
a rather different structure. In fact they do contain squared cycles $C^2_\ell$
for all $3\le\ell\le\sqc(n,\delta)$ with $\ell\neq 5$. If any one of these
`extra' squared cycles with chromatic number~$4$ is not present in the host
graph $G$, then~\ref{thm:main:ii} of Theorem~\ref{thm:main} guarantees even much
longer squared cycles $C^2_{\ell}$ in $G$, where~$\ell$ is a multiple of~$3$.
 
\begin{theorem} 
\label{thm:main} 
  For any $\nu>0$ there exists an integer $n_0$ such that for all integers 
  $n>n_0$ and $\delta\in[(\frac12+\nu)n,\frac23n]$ the 
  following holds for all $n$-vertex graphs $G$ with minimum degree $\delta(G)\ge \delta$.  
  \begin{enumerate}[label=\irom] 
    \item\label{thm:main:i} 
      $P^2_{\sqp(n,\delta)}\subset G$ and $C^2_{\ell}\subset G$
      for every $\ell\in\NATS$ with $3\le\ell\le \sqc(n,\delta)$ such that~$3$ divides~$\ell$.
    \item\label{thm:main:ii} 
      Either $C^2_\ell\subset G$ for every $\ell\in\NATS$ with $3\le\ell\le\sqc(n,\delta)$ 
      and $\ell\neq 5$, or $C^2_{\ell}\subset G$ for every $\ell\in\NATS$ with $3\le\ell\le 
      6\delta-3n-\nu n$ such that~$3$ divides~$\ell$.
  \end{enumerate} 
\end{theorem} 

The proof of this result relies on Szemer\'edi's Regularity Lemma\footnote{We
refer to~\cite{KOSurvey} for a survey on applications of the Regularity Lemma on
graph embedding problems.} and is presented together with the main lemmas in
Section~\ref{sec:proof}. Theorem~\ref{thm:main} cannot be extended to all values
of~$\delta(G)$ with $\delta(G)-\frac12n=o(n)$ because for infinitely many values
of~$m$ there are $C_4$-free graphs~$F$ on $m$ vertices with
$\delta(F)\ge\frac12\sqrt{m}$ (see~\cite{Rei}). Then, letting~$G$ be the
$n$-vertex graph obtained from~$F$ by adding an independent set~$I$
on~$m-\lfloor\frac12\sqrt{m}\rfloor$ vertices and inserting all edges
between~$F$ and~$I$, it is easy to see that $\delta(G)>\frac12n+\frac15\sqrt{n}$
but~$G$ does not contain a copy of $C^2_6$.
 
 
The following \emph{extremal graphs} show that the bounds in~\ref{thm:main:i}
and~\ref{thm:main:ii} of Theorem~\ref{thm:main} are tight (see also
Figure~\ref{fig:ext}). For~\ref{thm:main:ii} consider the complete tripartite
graph $K_{n-\delta,n-\delta,2\delta-n}$. Clearly, this graph has minimum
degree~$\delta$ and does not contain $C^2_\ell$ for any $\ell\ge 3$ not
divisible by~$3$ or $\ell\ge 3(2\delta-n)$. For the \emph{first part}
of~\ref{thm:main:i}, let~$G_p(n,\delta)$
\label{Gpndelta}
 be the $n$-vertex
graph obtained from the disjoint union of an independent  set~$Y$ on $n-\delta$
vertices and $r:=\rp(n,\delta)$ cliques $X_1,\ldots,X_r$
with~$|X_1|\le\dots\le|X_r|\le|X_1|+1$ on a total of~$\delta$ vertices, by
inserting all edges between~$Y$ and~$X_i$ for each $i\in[r]$. It is easy to
check that $\delta(G_p(n,\delta))=\delta$. Moreover any squared path
$P^2_m\subset G_p(n,\delta)$ contains vertices from at most one clique~$X_i$.
As~$Y$ is independent and~$P^2_m$ has independence number $\lceil m/3\rceil$ we
have $\lfloor2m/3\rfloor\le\lceil\delta/\rp(n,\delta)\rceil$ and thus
$m\le\lfloor\frac12(3\lceil\delta/\rp(n,\delta)\rceil+1)\rfloor=\sqp(n,\delta)$.
For the \emph{second part} of~\ref{thm:main:i}, we construct the graph
$G'_c(n,\delta)$ in the same way as~$G_p(n,\delta)$ but with~$r:=\rc(n,\delta)$
and with $|X_i|=\lceil\delta/r \rceil$ for \emph{all} $i\in[r]$. To obtain an
$n$-vertex graph $G_c(n,\delta)$ from~$G'_c(n,\delta)$ choose $v_i$ in $X_i$
arbitrarily for each $i\in[r]$ and identify all~$v_i$ with $i\le
r\lceil\delta/r\rceil-\delta$. Again $G_c(n,\delta)$ has minimum degree
$\delta$, any squared cycle~$C^2_m$ in $G_c(n,\delta)$ touches only one of
the~$X_i$, and hence $m\le\sqc(n,\delta)$. \begin{figure}[htbp] \centering
\psfrag{n-d}{\scalebox{.8}{$n-\delta$}} 
\psfrag{Gp}{$G_p(n,\delta)$} 
\psfrag{Gc}{$G_c(n,\delta)$} 
\psfrag{2d-n}{\scalebox{.8}{$2\delta-n$}} 
\psfrag{kttt}{$K_{n-\delta,n-\delta,2\delta-n}$} 
\includegraphics[width=\textwidth]{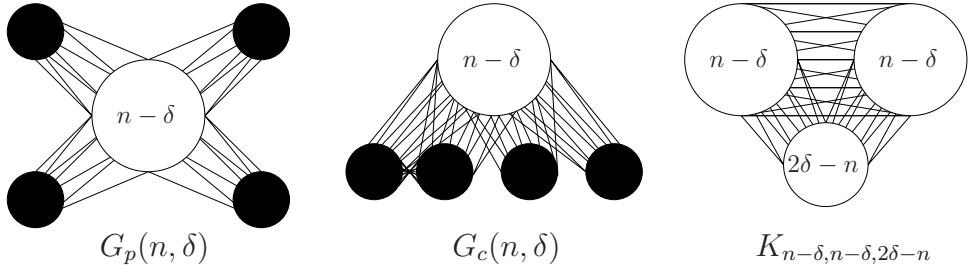} \caption{The extremal graphs,
for the  case
  $\rp(n,\delta)=\rc(n,\delta)=4$.}
\label{fig:ext} 
\end{figure} 
 
\smallskip
 
 
Before closing this introduction let us remark that similar phenomena to those
described in Theorem~\ref{thm:main} are observed with simple paths and cycles.
Every graph with minimum degree~$\delta$ contains a path of length $\lceil
n/\lfloor n/(\delta+1)\rfloor\rceil$, and the extremal graph is a vertex
disjoint union of cliques. This follows from an easy adjustment of the proof of
Dirac's theorem. Improving on results of Nikiforov and Schelp~\cite{NikiSchelp}
the first author proved the following theorem in~\cite{pathcycle}.  The methods
used for obtaining this result are quite different from those applied in this
paper. In particular they do not rely on the Regularity Lemma.
 
\begin{theorem}[Allen~\cite{pathcycle}]\label{pcycle}  
Given an integer $k\geq 2$ there is $n_0$ such that whenever $n\geq n_0$ and $G$ 
is an $n$-vertex graph with minimum degree $\delta\geq n/k$, the following are 
true. 
\begin{enumerate}[label=\irom] 
  \item $G$ contains $C_t$ for every even  
    $4\leq t\le\lceil n/(k-1)\rceil$, 
  \item\label{pcycle:missing}
    if $G$ does not contain a cycle of every length from 
    $\lfloor 2n/\delta\rfloor-1$ to 
    $\lceil n/(k-1)\rceil$ inclusive then $G$ does contain $C_t$ 
    for every even $4\le t\le 2\delta$. 
\end{enumerate} 
\end{theorem} 
 
\section{Main lemmas and proof of Theorem~\ref{thm:main}} 
\label{sec:proof}

Our proof of Theorem~\ref{thm:main} combines the Stability Method pioneered by
Simonovits~\cite{Sim}, the Regularity Method which pivots around the joint
application of Szemer\'edi's celebrated Regularity Lemma~\cite{Sze76}, and the
so-called Blow-up Lemma by Koml\'os, S\'ark\"ozy and Szemer\'edi~\cite{KSS_bl}.
The combination of these three methods has proved useful for a variety of exact
embedding results and was applied for example in~\cite{KSS_pos}.
However, this well-established technique provides only a rather loose
framework for proofs of this kind.
For our application we will embellish this framework with a new concept, which
we call the connected triangle components of a graph.

In this section we explain how we use connected triangle components, the
Regularity Method, and the Stability Method. We first provide the necessary
definitions, formulate our main lemmas (whose proofs are provided in the
remaining sections of this paper), and sketch how they work together in the
proof of Theorem~\ref{thm:main}. The details of this proof are then presented
at the end of this section. 

\smallskip

\paragraph{\bf Notation.} 

For a graph~$G$ we write $V(G)$ and $E(G)$ to
denote its vertex set and edge set, respectively, and set $v(G)=|V(G)|$, 
$e(G)=|E(G)|$ and $e(X,Y)=|\{xy\in E(G):x\in X, y\in Y\}|$ for sets $X,Y\subset
V(G)$. The graph $G[X]$ is the subgraph of~$G$ \emph{induced} by~$X$. The
\emph{neighbourhood} of a vertex~$v$ in~$G$ is denoted by~$\Gamma(v)$ and
$\Gamma(u,v)$ is the \emph{common neighbourhood} of $u,v\in V(G)$. For an edge
$uv=e\in E(G)$ we also write $\Gamma(e)=\Gamma(u,v)$. The minimum degree of
$G$ is denoted by $\delta(G)$ and for two sets $X,Y\subset V(G)$ we define
$\delta_Y(X)=\min_{x\in X}|\Gamma(x)\cap Y|$ and
$\delta_G(X)=\delta_{V(G)}(X)$.

When we say that a statement $\mathsf S(\epsilon,\epsilon')$ holds for 
positive real numbers
$
  \eps\gg\eps'
$,
then
we mean that, given an arbitrary $\eps>0$, we can find an $\epsilon''>0$ such
that $\mathsf S(\epsilon,\epsilon')$ holds for all
$\epsilon'\in(0,\epsilon'']$.

\smallskip

\paragraph{\bf Connected triangle components and triangle factors.}

Connected triangle components and connected triangle factors are the 
main protagonists in the proof of Theorem~\ref{thm:main}.
Roughly speaking, in a connected triangle component we can start in an
arbitrary triangle and reach each other triangle by ``walking'' through a
sequence of triangles, and a connected triangle factor is a collection of
vertex disjoint triangles each pair of which is connected in this way.

To make this precise, let $G=(V,E)$ be a graph. A \emph{triangle walk} in $G$ is
a sequence of edges $e_1,\dots,e_p$ in $G$ such that $e_i$ and $e_{i+1}$ share a
triangle of $G$ for all $i\in[p-1]$. We say that $e_1$ and $e_p$ are
\emph{triangle connected} in $G$. 
A \emph{triangle component} of $G$ is a maximal set of edges $C\subset E$ such
that every pair of edges in $C$ is triangle connected. Observe that this induces
an equivalence relation on the edges of $G$, but a vertex may be part of many
triangle components. In addition a triangle component does not need to form an
induced subgraph of $G$ in general.  The \emph{vertices of a triangle component}
$C_i$ are all vertices $v$ such that some edge $uv$ of $G$ is contained in
$C_i$.  We define the \emph{size $|C|$ of a triangle component} $C$ to be the
number of vertices of $C$.

A \emph{triangle factor}~$T$ in a graph~$G$ is a collection of vertex disjoint 
triangles in $G$. $T$ is a \emph{connected triangle factor} if all edges of~$T$
are in the same triangle component of~$G$. We define the \emph{size} of~$T$ to
be the number
of vertices covered by~$T$. We let $\CTF(G)$ denote the
maximum size of a connected triangle factor in $G$. It is not difficult to
check for example that any connected triangle factor in $G_p(n,\delta)$
contains only vertices of at most one of the cliques $X_i$ (cf.\ the
definition of $G_p(n,\delta)$ below Theorem~\ref{thm:main}) and of the
independent set~$Y$. Hence
\begin{equation}\label{eq:CTFGp}
  \CTF\big(G_p(n,\delta)\big) =
  3\left\lfloor\frac{\sqp(n,\delta)}{3}\right\rfloor .
\end{equation}
Suppose that a graph $G$ contains a square-path of length $\ell$. Then
obviously, $\CTF(G)\ge 3\lfloor \ell/3\rfloor$. Thus,~\eqref{eq:CTFGp}
together with Theorem~\ref{thm:main}\ref{thm:main:i} says that~$G_p(n,\delta)$
minimises $\CTF$ among all graphs of order $n$ and minimum degree $\delta$.

We will usually find that the number of vertices in a triangle component and
the size of a maximum connected triangle factor in that component are quite
different. As we will explain next, for the purposes of embedding squared paths
and squared cycles, it is the size of a connected triangle factor
that is important.

\smallskip

\paragraph{\bf The Regularity Method.}

The Regularity Lemma provides a partition of a dense graph that is suitable for
an application of the Blow-up Lemma, which is an embedding result for large host
graphs. In order to formulate the versions of these two lemmas that we will use,
we first introduce some terminology.
 
Let $G=(V,E)$ be a graph and $\eps,d\in(0,1]$. For disjoint nonempty $U,W\subset
V$ the \emph{density} of the pair $(U,W)$ is $d(U,W)=e(U,W)/|U||W|$. A pair
$(U,W)$ is \emph{$\eps$-regular} if
$|d(U',W')-d(U,W)|\le\eps$ for all $U'\subset U$ and $W'\subset W$ with
$|U'|\ge\eps|U|$ and $|W'|\ge\eps|W|$. An \emph{$\eps$-regular partition} of~$G$ is a
partition $V_0\dcup V_1\dcup\dots\dcup V_k$ of $V$ with $|V_0|\le\eps |V|$,
$|V_i|=|V_j|$ for all $i,j\in[k]$, and such that for all but at most $\eps
k^2$ pairs $(i,j)\in[k]^2$, the pair $(V_i,V_j)$ is $\eps$-regular.

Given some $0<d<1$ and a pair of
disjoint vertex sets $(V_i,V_j)$ in a graph $G$, we say that $(V_i,V_j)$ is
\emph{$(\eps,d)$-regular} if it is $\eps$-regular and has density at least $d$. We say
that an $\eps$-regular partition $V_0\dcup V_1\dcup\dots\dcup V_k$ of a graph $G$
is an \emph{$(\eps,d)$-regular partition} if the following is true. For every $1\leq
i\leq k$, and every vertex $v\in V_i$, there are at most $(\eps+d)n$ edges
incident to $v$ which are not contained in $(\eps,d)$-regular pairs of the
partition.

Given an $(\eps,d)$-regular partition $V_0\dcup V_1\dcup\dots\dcup V_k$ of a
graph $G$, we define a graph $R$, called the \emph{reduced graph} of the
partition of $G$, where $R=(V(R),E(R))$ has $V(R)=\{V_1,\ldots,V_k\}$ and
$V_iV_j\in E(R)$ whenever $(V_i,V_j)$ is an $(\eps,d)$-regular pair. We will
usually omit the partition, and simply say that $G$ has
\emph{$(\eps,d)$-reduced graph}~$R$. We call the partition classes $V_i$ with
$i\in[k]$ \emph{clusters} of $G$. Observe that our definition of the reduced
graph~$R$ implies that for $T\subset V(R)$ we can for example refer to the set
$\bigcup T$, which is a subset of $V(G)$.

The celebrated Szemer\'edi Regularity Lemma~\cite{Sze76} states that every
large graph has an $\eps$-regular partition with a bounded number of parts.
Here we state the so-called degree form of this lemma (see, e.g.,
\cite[Theorem~1.10]{KS96}).

\begin{lemma}[Regularity Lemma, degree form]\label{lem:SzReg}
  For every $\eps>0$ and every integer $m_0$, there is $m_1$ such that for every
  $d\in[0,1]$ every graph $G=(V,E)$ on $n\ge k_1$ vertices has an
  $(\eps,d)$-reduced graph~$R$ on $m$ vertices with $m_0\le m\le m_1$.
\end{lemma}

For our purpose it is convenient to work with even a different version of the
regularity lemma, which takes into account that we are dealing with graphs of
high minimum degree. This lemma is an easy corollary of
Lemma~\ref{lem:SzReg}. A proof can be found, e.g., in~\cite[Proposition
9]{KOTPlanar}.

\begin{lemma}[Regularity Lemma, minimum degree form]\label{lem:reg}
 For all $\eps$, $d$, $\gamma$ with $0<\eps<d<\gamma< 1$ and for every
 $m_0$, there is $m_1$ such that every
 graph $G$ on $n>m_1$ vertices with $\delta(G)\ge\gamma n$ has an
 $(\eps,d)$-reduced graph~$R$ on $m$ vertices with $m_0\le m\le m_1$ and
 $\delta(R)\ge(\gamma-d-\eps)m$.
\end{lemma}

This lemma asserts that the reduced graph~$R$ of~$G$ ``inherits'' the high
minimum degree of~$G$. We shall use this
property in order to reduce the original problem of finding a squared
path (or cycle) in an $n$-vertex graph with minimum degree $\gamma n$ to the
problem of finding an \emph{arbitrary} connected triangle factor of a certain size in an
$m$-vertex graph~$R$ with minimum degree $(\gamma-d-\eps)m$. The new problem 
is much less particular about the required subgraph than the original one and
hence easier to attack (see Lemma~\ref{StabLem}).

This kind of reduction is made possible by the Blow-up Lemma. 
Roughly, this lemma asserts that a bounded degree graph~$H$ can be embedded into
a graph~$G$ with reduced graph~$R$ if there is a homomorphism from~$H$ to a
subgraph $S$ of~$R$ which does not ``overfill'' any of the clusters in~$S$. In
our setting we apply this lemma with~$S=K_3$ and conclude that for \emph{each}
triangle~$t$ of a connected triangle factor~$T$ in~$R$ we find a squared path
in~$G$ that almost fills the clusters of~$G$ corresponding to~$t$. By using
the fact that~$T$ is triangle \emph{connected} it is then possible to connect these
squared paths into squared paths or cycles of the desired overall length. 
In addition, the Blow-up Lemma allows for some control about the start- and
end-vertices of the path that is constructed in this way (cf.\
Lemma~\ref{lem:bl}\ref{lem:bl:3}).

The following lemma summarises this embedding technique, which is also implicit, e.g.,
in~\cite{KSS_pos}. For completeness we provide a proof of this lemma in the
appendix.
 
\begin{lemma}[Embedding Lemma]
\label{lem:bl}
For all $d>0$ there exists $\eps_\subsc{EL}>0$ with the following property.
Given $0<\eps<\eps_\subsc{EL}$, for every $m_\subsc{EL}\in\mathbb{N}$ there
exists $n_\subsc{EL}\in\mathbb{N}$ such that the following hold for any graph $G$ on $n\ge n_\subsc{EL}$ vertices with
$(\eps,d)$-reduced graph $R'$ on $m\le m_\subsc{EL}$ vertices.
\begin{enumerate}[label=\irom]
  \item\label{lem:bl:1}
    $C^2_{3\ell}\subset G$ for every $\ell\in\NATS$ with $3\ell\le
    (1-d)\CTF(R')\frac{n}{m}$.
  \item\label{lem:bl:2}
    If $K_4\subset C$ for each triangle component~$C$ of~$R'$, then
    $C^2_\ell\subset G$ for every $\ell\in\NATS\setminus\{5\}$ with
    $3\le\ell\le (1-d)\CTF(R')\frac{n}{m}$.
\end{enumerate}
Furthermore, let $T$ be a connected triangle factor in a triangle component~$C$
of~$R$ with $K_4\subset C$, let $u_1v_1,u_2v_2\in E(G)$ be disjoint edges, and
suppose that there are (not necessarily disjoint) edges $X_1Y_1, X_2Y_2\in C$
such that the edge $u_iv_i$ has at least $2d\frac{n}{m}$ common neighbours in
each cluster $X_i$ and $Y_i$ for $i=1,2$. Then
\begin{enumerate}[label=\irom,start=3]
  \item\label{lem:bl:3}
    $P^2_{\ell}\subset G$ for every $\ell\in\NATS$ with
    $6(m+2)^3<\ell<(1-d)|T|\frac{n}{m}$, such that $P^2_{\ell}$ starts in
    $u_1,v_1$ and ends in $u_2,v_2$ (in those orders) and at most $(\eps+d)n$
    vertices of $P^2_{\ell}$ are not in $\bigcup T$.
\end{enumerate}
\end{lemma} 

The copies of~$K_4$ that are required in this lemma play a crucial r\^ole when
embedding squared cycles which are not $3$-chromatic.
 
\smallskip

\paragraph{\bf The Stability Method.}

The strategy we just described leaves us with the task of finding a big
connected triangle factor~$T$ in the reduced graph~$R$ of~$G$.
However, there is one problem with this approach: 
The proportion~$\tau$ of~$R$ covered by~$T$ is roughly equal to the
proportion of~$G$ covered by the squared path~$P$ that we obtain from the
Embedding Lemma (Lemma~\ref{lem:bl}). However, as explained above, the relative
minimum degree $\gamma_R=\delta(R)/|V(R)|$ of~$R$ is in general slightly smaller
than~$\gamma_G=\delta(G)/|V(G)|$, but the extremal graphs for squared paths and
connected triangle factors are the same. It follows that we cannot expect
that~$\tau$ is larger than the proportion a maximum squared path covers in a
graph with relative minimum degree~$\gamma_R$, and hence smaller than the
proportion we would like to cover for relative minimum degree~$\gamma_G$.

Consequently we need to be more ambitious and shoot for a bigger connected
triangle factor in~$R$ than we can expect for this minimum degree (cf.\
Lemma~\ref{StabLem}~\ref{StabLem:1} and~\ref{StabLem:2}). This will of course
not always be possible, but it will only fail if~$R$ (and hence~$G$) is `very
close' to the extremal graph $G_p(|V(R)|,\delta(R))$ (and hence also to
$G_c(|V(R)|,\delta(R))$) in which case we will say that~$R$ is
\emph{near-extremal} (cf.\ Lemma~\ref{StabLem}~\ref{StabLem:3}).

This approach is called the Stability Method and the following lemma states that it
is feasible for our purposes. This lemma additionally guarantees copies of~$K_4$ as
required by the Embedding Lemma. We formulate this lemma for graphs~$G$, but
use it on the reduced graph~$R$ later. Its proof does not rely on the
Regularity Lemma and is given in Section~\ref{sec:comp}.

\begin{lemma}[Stability Lemma]
\label{StabLem} 
  Given $\mu>0$, for any sufficiently small $\eta>0$ there exists $n_0$ such
  that if $G$ has $n>n_0$ vertices and
  $\delta(G)=\delta\in((\frac12+\mu)n,\frac{2n-1}{3})$, then either
  \begin{enumerate}[label={\itm{S\arabic{*}}}]
    \item\label{StabLem:1} $\CTF(G)\ge 3(2\delta-n)$, or 
    \item\label{StabLem:2}  
      $\CTF(G)\ge\min( \sqp(n,\delta+\eta n), \frac{11n}{20})$, 
      or 
    \item\label{StabLem:3} 
      $G$ has an independent set of size at least $n-\delta-11\eta n$ whose
      removal disconnects $G$ into  
      components, each of size at most $\frac{19}{10}(2\delta-n)$. 
  \end{enumerate} 
  Moreover, in cases~\ref{StabLem:2} and~\ref{StabLem:3} 
  each triangle component of $G$ contains a $K_4$. 
\end{lemma} 

By the discussion above, it remains to handle the graphs with near-extremal
reduced graph. For these graphs we have a lot of structural information which
enables us to show directly that they contain the squared paths and squared
cycles we desire, as the following lemma documents. The proof of this lemma is
provided in Section~\ref{sec:ext}.  In this proof we shall again make use of the
embedding lemma, Lemma~\ref{lem:bl}. Accordingly Lemma~\ref{lem:ext-like}
inherits the upper bound $m_\subsc{EL}$ on the number of clusters from
Lemma~\ref{lem:bl}.

\begin{lemma}[Extremal Lemma]
\label{lem:ext-like} 
 For every $\nu>0$, given $0<\eta,d<10^{-8}\nu^4$ there exists $\eps_0>0$
 such that for every $0<\eps<\eps_0$ and every $m_\subsc{EL}$, there exists
 $N$ such that the following holds. Suppose that \begin{enumerate}[label={\irom}]
    \item\label{ext:assm1}
      $G$ is an $n$-vertex graph with $n>N$ and $\delta(G)=
      \delta>\frac{n}{2}+\nu n$,
    \item\label{ext:assm2}
      $R$ is an $(\varepsilon,d)$-reduced graph of~$G$ of order $m\le
      m_\subsc{EL}$,
    \item\label{ext:assm3}
      each triangle component of $R$ contains a copy of $K_4$.
    \item\label{ext:assm4}
      $V(R)=I \dcup B_1 \dcup B_2\dcup\cdots\dcup B_k$ with
      $k\geq 2$,
    \item\label{ext:assm5}
      $I$ is an independent set with $|I|\ge(n-\delta-11\eta n)m/n$,
    \item\label{ext:assm6}
      for each $i\in[k]$ we have $0<|B_i|\le 19m(2\delta-n)/(10n)$, and for every
      $j\in[k]\setminus\{i\}$ there are no edges between $B_i$ and $B_j$ in $R$.
   \end{enumerate} 
   Then $G$ contains $P^2_{\sqp(n,\delta)}$ and $C^2_{\ell}$ for each
   $\ell\in[3,\sqc(n,\delta)]\setminus\{5\}$.
\end{lemma} 

It is interesting to notice that, although the two functions $\sqp(n,\delta)$
and $\sqc(n,\delta)$ are different---their jumps as $\delta$ increases occur
at slightly different values---they are similar enough that the Stability
Lemma covers them both. We will only need to distinguish between squared paths
and squared cycles when we examine the near-extremal graphs.

\smallskip

\paragraph{\bf Proof of Theorem~\ref{thm:main}.} 

With this we have all the ingredients for the proof of our main theorem,
which uses the Regularity Lemma (in form of Lemma~\ref{lem:reg}) to
construct a regular partition with reduced graph~$R$ of the host graph~$G$,
the Stability Lemma (Lemma~\ref{StabLem}) to conclude that~$R$ either
contains a big connected triangle factor or is near-extremal, the Embedding
Lemma (Lemma~\ref{lem:bl}) to find long squared paths and cycles in~$G$ in
the first case, and the Extremal Lemma (Lemma~\ref{lem:ext-like}) in the
second case.

\begin{proof}[Proof of Theorem~\ref{thm:main}]
We require our constants to satisfy 
\[\nu\gg\mu\gg\eta\gg d\gg \varepsilon>0\,,\] 
which we choose, given $\nu$, as follows. First, we choose $\mu:=\nu/2$. We then
choose $\eta>0$ to be small enough for both Lemma~\ref{StabLem} and
Lemma~\ref{lem:ext-like}. Now we set $d>0$ to be small enough for
Lemma~\ref{lem:ext-like} and such that $d\le\nu/10$ and $d\le \eta/10$. 
For this~$d$ Lemma~\ref{lem:bl} then produces a constant~$\eps_\subsc{EL}$.
We choose $\eps>0$ to be smaller than
$\min\{\eps_\subsc{EL},\nu/10\}$ and sufficiently small for Lemma~\ref{lem:ext-like}. 
We choose $m_0$ to be sufficiently large to apply Lemma~\ref{StabLem} to
any graph with at least $m_0$ vertices. We then choose $m_\subsc{EL}$ such that
Lemma~\ref{lem:reg} guarantees the existence of an $(\eps,d)$-regular partition
with at least $m_0$ and at most $m_\subsc{EL}$ parts. Finally we choose
$n_0>n_\subsc{EL}$ to be sufficiently large for both Lemma~\ref{lem:bl} and
Lemma~\ref{lem:ext-like}.

Let $n>n_0$ and $\delta\in(n/2+\nu n,n-1]$. Let $G$ be any $n$-vertex graph with
$\delta(G)\geq\delta$.  Observe that it suffices to show that
$P^2_{\sqp(n,\delta)}\subset G$ and that~\ref{thm:main:ii} of
Theorem~\ref{thm:main} holds.  We first apply Lemma~\ref{lem:reg} to $G$ to
obtain an $(\varepsilon,d)$-reduced graph~$R$ on $m_0\le m\le m_\subsc{EL}$
vertices. Let $\delta'=\delta(R)\geq (\delta/n-d-\varepsilon)m>m/2+\mu m$.  Then we apply
Lemma~\ref{StabLem} to~$R$. There are three possibilities.
 
First, we could find that $\CTF(R)\geq 3(2\delta'-m)$. In this case by
Lemma~\ref{lem:bl} we are guaranteed that for every integer~$\ell'$ with
$3\ell'<(1-d)\CTF(R)n/m$ we have $C^2_{3\ell'}\subset G$. By choice of~$d$
and~$\varepsilon$ we have $(1-d)\cdot 3(2\delta'-m)n/m>6\delta-3n-\nu n$. Noting
that $P^2_\ell\subset C^2_\ell$ we conclude that $P^2_{\sqp(n,\delta)}\subset G$
and $C^2_{\ell}\subset G$ for each integer $\ell\leq 6\delta-3n-\nu n$ such
that~$3$ divides~$\ell$, i.e., the second case of
Theorem~\ref{thm:main}\ref{thm:main:ii} holds. 
 
Second, we could find that $\CTF(R)\geq\min(\sqp(m,\delta'+\eta
m),\frac{11m}{20})$ and that every triangle component of $R$ contains a copy of
$K_4$. By Lemma~\ref{lem:bl} we are guaranteed that for every
$\ell\in[6,(1-d)\CTF(R)n/m]\setminus\{5\}$ we have $C^2_\ell\subset G$. By
choice of $\eta$ and $d$ we have
$(1-d)\CTF(R)n/m>\sqp(n,\delta)\ge \sqc(n,\delta)$, so we have $P^2_{\sqp(n,\delta)}\subset G$ and for each integer
$\ell\in[3,\sqc(n,\delta)]\setminus\{5\}$ we have $C^2_{\ell}\subset G$, i.e.,
the first case of Theorem~\ref{thm:main}\ref{thm:main:ii} holds.
 
Third, we could find that $R$ is near-extremal. Then $R$ contains an independent
set on at least $m-\delta'-11\eta m$ vertices whose removal disconnects $R$ into
components of size at most $\frac{19}{10}(2\delta'-m)$, and each
triangle component of $R$ contains a copy of $K_4$. But now $G$ satisfies the
conditions of Lemma~\ref{lem:ext-like}. It follows that $G$ contains
$P^2_{\sqp(n,\delta)}$ and for each $\ell\in[3,\sqc(n,\delta)]\setminus\{5\}$
the graph $G$ contains $C^2_\ell$, i.e., the first case of
Theorem~\ref{thm:main}\ref{thm:main:ii} holds.
\end{proof}

\section{Triangle Components and the proof of Lemma~\ref{StabLem}} 
\label{sec:comp} 
 
In this section we provide a proof of our stability result for connected triangle
factors, Lemma~\ref{StabLem}. 
%
%
Distinguishing different cases, we analyse the sizes and the structure of the
triangle components in the graph~$G$ under study. Before we give more details
about our strategy and a sketch of the proof, we introduce some additional
definitions and provide a preparatory lemma (Lemma~\ref{lem:cpts}).

Let $G$ be a graph with triangle components $C_1,\ldots,C_r$. The \emph{interior}
$\intr(G)$ of $G$ is the set of vertices of $G$ which are in more than one of the
triangle components. For a component $C_i$, the interior of $C_i$, written
$\intr(C_i)$, is the set of vertices of $C_i$ which are in $\intr(G)$. The
remaining vertices of $C_i$ are called the \emph{exterior} $\bound(C_i)$. That
is, $\bound(C_i)$ is formed by the set of vertices of $C_i$ which are in no other
triangle component of $G$. To give an example, by definition the graph
$G_p(n,\delta)$ has $\rp(n,\delta)$ triangle components; its interior is the
independent set $Y$ (using the notation of the construction of~$G_p(n,\delta)$
on page~\pageref{Gpndelta} in
Section~\ref{sec:Intro}), with the component exteriors being the cliques
$X_1,\ldots,X_r$. 

The following lemma collects some observations about triangle
components.
 
\begin{lemma}\label{lem:cpts} 
  Let~$G$ be an $n$-vertex graph with~$\delta(G)=\delta>n/2$. 
  Then  
  \begin{enumerate}[label=\iabc] 
    \item\label{lem:cpts:a} each triangle component $C$ of $G$ satisfies 
      $|C|>\delta$, 
    \item\label{lem:cpts:c} for distinct triangle components $C$, $C'$ we have 
    $e(\bound(C),\bound(C'))=0$, 
    \item\label{lem:cpts:d} for each triangle component $C$, 
      each vertex $u$ of $C$, and $U:=\{v\colon uv\in C\}$, 
      the minimum degree in $G[U]$ is at least $2\delta-n$ and hence 
      $|G[U]|\ge 2\delta-n+1$. 
  \end{enumerate} 
\end{lemma} 
\begin{proof} 
  To see~\ref{lem:cpts:a} let $M$ be the vertices of a maximal clique in $C$ (clearly $|M|\ge 
  3$). If $u$ and $v$ are in $M$, and $x$ is a common neighbour 
  of $u$ and~$v$, then $x$ is also in $C$. Thus vertices of $G\setminus C$ are 
  adjacent to at most $1$ vertex of $M$ and vertices of $C$ are adjacent to 
  at most $|M|-1$ vertices of $M$, by maximality of~$M$. This gives the inequality 
  \[|M|\delta\leq \sum_{m\in M}d(m)\leq \sum_{x\in C}(|M|-1)+\sum_{x\notin C}1\]  
  and hence $|M|\delta-n\leq (|M|-2)|C|$. Since $n<2\delta$ we have 
  $|C|>\delta$ as required. 
  
  Since $\delta>n/2$, we have that
  $\Gamma(u,u')\neq\emptyset$ for any two vertices $u$ and~$u'$. Now, if
  $u\in\bound(C)$, $u'\in\bound(C')$, $x\in\Gamma(u,u')$, and $uu'$ was an edge, then $uu'x$
  would form a triangle. Then $u$ and $u'$
  would be together in some triangle component $C''$, contradicting the fact that they
  are in the exterior. Therefore, the assertion~\ref{lem:cpts:c} follows. 

  Moreover, for an edge $uv$ of $C$ we have $\Gamma(u,v)\subset C$ as $C$ is a
  triangle component. Since $|\Gamma(u,v)|\ge2\delta-n$ we get~\ref{lem:cpts:d}.
\end{proof} 
 
Now let us sketch the proof of Lemma~\ref{StabLem}.
Lemma~\ref{lem:cpts}\ref{lem:cpts:a} states that triangle components cannot be
too small. However, it is not solely the size of the triangle components we are
interested in: we want to find a triangle component that contains many vertex
disjoint triangles. At this point, Lemma~\ref{lem:cpts}\ref{lem:cpts:d} comes
into play. It asserts that certain spots in a triangle component induce a graph
with minimum degree $2\delta-n$. In the proof of Lemma~\ref{StabLem} we shall
usually (i.e., for many values of~$\delta$) use this fact in order to find a big
matching~$M$ in such spots (Proposition~\ref{prop:match}\ref{prop:match:a} below
asserts that this is possible). Clearly all edges in such a matching are
triangle connected and hence it will remain to extend~$M$ to a set of vertex
disjoint triangles. For this purpose we will analyse the size of the common
neighbourhood $\Gamma(u,v)$ of an edge $uv$ in~$M$. We will usually find that
$\Gamma(u,v)$ is so big that a simple greedy strategy allows us to construct the
triangles. For estimating $\Gamma(u,v)$ we will often use the following
technique: We find a large set $X$ such that neither $u$ nor $v$ has neighbours
in $X$. This implies $|\Gamma(u,v)|\ge 2\delta-(n-|X|)$. Observe that
Lemma~\ref{lem:cpts}\ref{lem:cpts:c} implies that $\bound(C)$ can serve as $X$
if both $u$, $v\in\bound(C')$ for some triangle components $C$ and $C'$.
 
The strategy we just described works for most values of $\delta$ below $\frac35
n$ (we describe the exceptions below). For $\delta\ge\frac35 n$ however, the
greedy type argument fails, the reason being that we usually bound the common
neighbourhood of an edge used in the argument above by $4\delta-2n$.  But for
$\delta\ge\frac35 n$ we might have $\sqp(n,\delta)>4\delta-2n$ (see
Figure~\ref{fig:sqpplot}). We solve this problem by using a different strategy
in this range of $\delta$. We will still start with a big connected matching~$M$
as before, but use a Hall-type argument to extend~$M$ to a triangle factor
$T$. More precisely, we find $M$ in the exterior of some triangle component and
then consider for each edge $uv$ of $M$ all common neighbours of $uv$ in
$\intr(G)$. The Hall-type argument then permits us to find distinct extensions
for the edges of $M$. To make this argument work we use the fact that in this
range of $\delta$ the set $\intr(G)$ is an independent set.
 
We indicated earlier that there are some exceptional values of $\delta$ that
require special treatment: namely $\delta$ close to $\frac35 n$ and $\frac47
n$. Observe that in both ranges the number of triangle components of
$G_p(n,\delta)$ changes (from $2$ to $3$ for $\frac35 n$, and from $3$ to $4$
for $\frac47n$) and thus the value $\sqp(n,\delta)$ as a function in $\delta$
jumps (see Figure~\ref{fig:sqpplot}). Roughly speaking, the reason that these
two ranges need to be treated separately is that again $\sqp(n,\delta)$ is not
substantially smaller than $4\delta-2n$ here, but we also do not know now that
$\intr(G)$ is an independent set. For dealing with these values of $\delta$ we
will use a somewhat technical case analysis which we provide at the end of this
section (as proof of Fact~\ref{fac:StabLem:int:special}).
 
As explained above, we will apply the following simple observations about
matchings in graphs of given minimum degree. 
 
\pagebreak[4]
\begin{proposition}\label{prop:match} \mbox{} \\[-5mm]
  \begin{enumerate}[label=\iabc]
  \item\label{prop:match:a} Let $G=(X,E)$ be a graph with minimum
    degree~$\delta$. Then~$G$ has a matching covering
    $2\min(\delta,\lfloor|X|/2\rfloor)$ vertices.
  \item\label{prop:match:b} Let $G=(A\dcup B,E)$ be a bipartite graph 
    with parts $A$ and $B$, 
    such that every vertex in $A$ has degree at least $a$ and every vertex in $B$ has
    degree at least $b$. Then~$G$ has a matching covering $2\min(a+b,|A|,|B|)$ vertices.
  \end{enumerate}
\end{proposition} 
\begin{proof}
  We first prove~\ref{prop:match:a}. Let $M$ be a maximum matching in $G$, and assume that $M$ 
  contains less than $\min(\delta,\lfloor|X|/2\rfloor)$ edges. In particular,
  there are two vertices $x,y\in X$ not covered by $M$. Clearly, all neighbours
  of $x$ and $y$ are covered by~$M$. 
  
  We claim that there is an edge $uv$ in $M$ with
  $xu,yv\in E$. Indeed, suppose that this is not the case. 
  Then  $|e\cap \Gamma(x)|+|e\cap \Gamma(y)|\le 2$ for each $e\in M$. We
  therefore have
   \[\delta+\delta\le |\Gamma(x)|+|\Gamma(y)|= \sum_{e\in M}(|e\cap
   \Gamma(x)|+|e\cap \Gamma(y)|)\le 2|M|\,,\]
  contradicting the fact that $\delta>|M|$.
  
  Now, let $uv\in M$ be an edge as in  the claim above. Since $xu,yu\in E$ we
  get that $x,u,v,y$ is an
  $M$-augmenting path, a contradiction.
  
  Next we prove~\ref{prop:match:b}. Let $M$ be a maximum matching in $G$. We are
  done unless there are vertices $u\in A$ and $v\in B$ not contained in
  $M$. There cannot be an edge $xy\in M$ such that $uy$ and $xv$ are edges of
  $G$ by maximality of $M$, since then $u,y,x,v$ was an $M$-augmenting path. But
  $u$ has at least $a$ neighbours in $V(M)\cap B$, and $v$ at least $b$
  neighours in $V(M)\cap A$, so there must be at least $a+b$ edges in $M$.
\end{proof} 
 
Before turning to the proof of Lemma~\ref{StabLem} let us quickly collect some 
analytical data about $\sqp(n,\delta)$ and $\rp(n,\delta)=:r$. It is not
difficult to check that
\begin{equation}\label{eq:r} 
\begin{split}
  \frac{(r+1)n-r}{2(r+1)-1}&\le\delta<\frac{rn-r+1}{2r-1} \quad\text{and}\quad \\
  \frac{n-\delta}{2\delta-n+1}&\le r<\frac{\delta+1}{2\delta-n+1}\,. 
\end{split}
\end{equation} 
For the proof of Lemma~\ref{StabLem} it will be useful to note in addition that
given $\mu>0$, for every $0<\eta<\eta_0=\eta_0(\mu)$, there is $n_1=n_1(\eta)$
such that the following holds for all $n\geq n_1$.  For all
$\delta,\delta'>\frac{n}{2}+\mu n$, where~$\delta$ is such that
$\sqp(n,\delta+\eta n)\le\frac{11}{20}n$, and where~$\delta'$ is such that we
have $\rp(n,\delta')\ge 3$ and either $\rp(n,\delta')\ge 5$ or
$\rp(n,\delta')=\rp(n,\delta'+\eta n)$, we have
\begin{equation}\label{eq:sqpa} 
  \sqp(n,\delta+\eta n) 
    \le\frac32\min\Big(\frac{\delta}{\rp(n,\delta+\eta n)-1}-2, \,
      \frac{\delta+3\eta n}{\rp(n,\delta+\eta n)}-2\Big),
\end{equation}
\begin{equation}\label{eq:sqpb}
 \begin{split}
  \sqp(n,\delta+\eta n) 
    & \le \tfrac{19}{20}\cdot 3(2\delta-n)-2
    \le 6\delta-3 n-100\eta n,
  \quad\text{and}\quad \\
  \sqp(n,\delta'+\eta n) 
    &\le4\delta'-2 n, 
 \end{split}
\end{equation} 
which follows immediately from the definition of $\sqp(n,\delta)$
in~\eqref{eq:sqpc} (see also Figure~\ref{fig:sqpplot}).
  
\begin{proof}[Proof of Lemma~\ref{StabLem}] 
  Given $\mu$ and any $0<\eta<\min(\frac{1}{1000},\eta_0(\mu),2\mu^2/3)$, where
  $\eta_0(\mu)$ is as above~\eqref{eq:sqpa}, let $n_0:=\max(n_1(\eta),2/\eta)$
  with $n_1(\eta)$ as above~\eqref{eq:sqpa}.  Let $n\ge n_0$. This in particular
  means that we may assume the inequalities~\eqref{eq:sqpa} and~\eqref{eq:sqpb}
  in what follows. Define %
  $\gamma:=\delta/n$, and $r:=\rp(n,\delta)$ and $r':=\rp(n,\delta+\eta n)$.
 
  If $G$ has only one triangle component then Theorem~\ref{thm:CorHaj} guarantees 
  that $\CTF(G)\ge 6\delta-3n$  
  and so we are in Case~\ref{StabLem:1}. Thus we may assume in 
  the following that $G$ has at least two triangle
  components. Then Lemma~\ref{lem:cpts}\ref{lem:cpts:a}
  implies that $\intr(C)\not=\emptyset$ for any
  triangle component~$C$.
 
  Suppose that $C$ is a triangle component of $G$ which does not contain a copy 
  of $K_4$. Let $u$ be a vertex of $C$, and $U:=\{v\colon uv\in C\}$. By 
  Lemma~\ref{lem:cpts}\ref{lem:cpts:d} 
  we have $\delta(G[U])\geq 
  2\delta-n$. Because $C$ contains no copy of $K_4$, $U$ contains no triangle. 
  By Tur\'an's theorem we have $|U|\geq 2(2\delta-n)$, and so by
  Proposition~\ref{prop:match}\ref{prop:match:a} the set $U$ contains a
  matching $M$ with $2\delta-n$ edges. Finally we choose greedily for each $e\in M$ a distinct
  vertex $v\in V(G)$ such that $ev$ is a triangle. Since $U$ is triangle free
  all these vertices must lie outside $U$, and since $|\Gamma(e)|\ge
  2\delta-n$ we cannot fail to find distinct vertices for each edge. This
  yields a set $T$ of $2\delta-n$ vertex-disjoint triangles which are all in
  $C$. So $\CTF(G)\geq 6\delta-3n$ and we are in case~\ref{StabLem:1}.
  Henceforth we assume that every triangle component of~$G$ contains a copy of
  $K_4$.
 
  We continue by considering the case $\frac{3n-2}{5}\leq\delta<\frac{2n-1}{3}$.
  The following observation readily implies the lemma in this range, as we will 
  see in Fact~\ref{fac:StabLem:35}. 
   
  \begin{fact} 
  \label{fac:StabLem:around35} 
    If 
    $\delta(G)\ge(\frac35-2\eta)n$, 
    $G$ has exactly $2$ triangle components,  
    $\intr(G)$ is independent, 
    and either $|\intr(G)|<n-\delta-11\eta n$ or the exterior $X$ of the  
    triangle component with most vertices satisfies
    $|X|\ge\frac{19}{10}(2\delta-n)$, then $\CTF(G)\ge\min(\sqp(n,\delta+\eta n),\frac{11}{20}n)$. 
  \end{fact}

  \begin{factproof}[Proof of Fact~\ref{fac:StabLem:around35}]
    First, by Lemma~\ref{lem:cpts}\ref{lem:cpts:c} a vertex $x\in
    X$ cannot have neighbours in the exterior of the other triangle component, so
    $\Gamma(x)\subset X\cup\intr(G)$. Thus $\delta(G[X])\ge\delta-|\intr(G)|$,
    which by Proposition~\ref{prop:match}\ref{prop:match:a} means that there is
    a matching $M$ in $G[X]$ with
    \begin{equation}\label{eq:fact1:M}
      |M|=\min(\delta-|\intr(G)|,\lfloor|X|/2\rfloor) 
    \end{equation}
    edges.

    We aim to pair off edges of $M$ with vertices of $\intr(G)$ to form a
    sufficiently large number of vertex-disjoint triangles. To see that a
    triangle factor resulting from this process will be connected, observe that
    all edges of $M$ are in $X$, and since $X$ is a triangle component exterior,
    the edges of $M$ are triangle connected. To form triangles from edges of $M$
    and vertices of $\intr(G)$, we introduce an auxiliary bipartite graph $H$
    with vertex set $M\dcup\intr(G)$, where $uv\in M$ is adjacent in $H$ to
    $w\in\intr(G)$ iff $uvw$ is a triangle of $G$. Every vertex of $X$ has at
    least $\delta-|X|$ neighbours in $\intr(G)$, and so every edge of $M$ has at
    least $a:=2(\delta-|X|)-|\intr(G)|$ common neighbours in $\intr(G)$. At the
    same time, since $\intr(G)$ is independent, every vertex of $\intr(G)$ has
    at least $\delta-(n-|\intr(G)|-|X|)$ neighbours in $X$, of which all but
    $|X|-2|M|$ must be in $M$. So every vertex of $\intr(G)$ must have at least
    \[b:=\delta-(n-|\intr(G)|-|X|)-(|X|-2|M|)-|M|=\delta-n+|\intr(G)|+|M|\]
    edges of $M$ in its neighbourhood. By
    Proposition~\ref{prop:match}\ref{prop:match:b} there is a matching in $H$ on
    at least $\min(a+b,|M|,|\intr(G)|)$ edges, and hence a connected triangle
    factor in $G$ with so many triangles. Observe that
    \begin{equation}\label{eq:fact1:aplusb}
    \begin{split}
      a+b &= 2\delta-2|X|-|\intr(G)|+\delta-n+|\intr(G)|+|M| \\
      &= 3\delta-n-2|X|+|M|\,.
    \end{split}
    \end{equation}

    Since there are two triangle components in~$G$, there is a vertex~$u$ in a triangle
    component exterior which is not $X$. Therefore~$u$ has no neighbour in
    $X$, so $|X|<n-\delta$. Since $\delta\ge(\frac35-2\eta)n$,
    by~\eqref{eq:fact1:aplusb} we have 
    \begin{equation}\label{eq:fact1:aplusb2}
      a+b> |M|-10\eta n \,.
    \end{equation}  
    Furthermore, 
    \begin{equation}\label{eq:fact1:aplusb3}
      \quad\text{if}\quad
     |X|\le(\tfrac{2}{5}-3\eta)n\,, 
     \quad\text{then}\quad
     a+b\ge |M| \,.
   \end{equation}
    By Lemma~\ref{lem:cpts}\ref{lem:cpts:a} we have
    $|\intr(G)|\ge2\delta-n\ge\frac{n}{5}-4\eta n$. Since
    $\eta\le\frac{1}{1000}$ we have
    \[3|\intr(G)|\ge \frac{3n}{5}-12\eta n>\frac{11n}{20}\,.\] Thus we have
    $\CTF(G)\ge\frac{11n}{20}$ if we find a matching in $H$ covering
    $\intr(G)$. It remains, then, to check that 
    we have
    \begin{equation}\label{eq:fact1:min}
      3\min(a+b,|M|)\ge\min(\sqp(n,\delta+\eta n),\frac{11}{20}n).
    \end{equation}
    We distinguish two cases.

    {\sl Case 1:} $a+b<|M|$. By~\eqref{eq:fact1:aplusb3} this 
    forces $|X|>(\frac{2}{5}-3\eta)n$. 
    Since we have $|M|=\min(\delta-|\intr(G)|,\lfloor|X|/2\rfloor)$ by~\eqref{eq:fact1:M},
    there are two possibilities.  If $|M|=\lfloor|X|/2\rfloor$ then we
    have 
    \[
    a+b
    \geByRef{eq:fact1:aplusb2}
    \Big\lfloor\frac{|X|}{2}\Big\rfloor-10\eta n>\frac{n}{5}-12\eta
    n>\frac{11n}{60}\,,\]
    which proves~\eqref{eq:fact1:min} in this subcase. If, on the other hand,
    $|M|=\delta-|\intr(G)|$, then we use that $\intr(G)$ is independent, which
    implies $\intr(G)\le n-\delta$ and thus
    \begin{equation*}\begin{split}
      a+b
      & \geByRef{eq:fact1:aplusb2} |M|-10\eta n
      =\delta-|\intr(G)|-10\eta n
      \ge 2\delta-n-10\eta n \\
      & \geByRef{eq:sqpb}\tfrac13\sqp(n,\delta+\eta n)
      \,,
    \end{split}\end{equation*}
    which proves~\eqref{eq:fact1:min} in this subcase.

    {\sl Case 2:} $a+b\ge |M|$. In this case, $H$ contains a matching of
    size~$|M|$, so we have $\CTF(G)\ge 3|M|=3\min(\delta-|\intr(G)|,\lfloor|X|/2\rfloor)$.  Again there are two
    possibilities, depending on $|M|$. If $|M|=\delta-|\intr(G)|$, we are done by
    \eqref{eq:sqpb} exactly as before. If, on the other hand,
    $|M|=\lfloor|X|/2\rfloor$, then~\eqref{eq:fact1:min} holds (and hence we are
    done) unless
    \begin{equation}\label{eq:StabLem:X}
      3\lfloor\tfrac{|X|}{2}\rfloor <\min\big(\sqp(n,\delta+\eta
      n),\tfrac{11}{20}n\big)\,.
    \end{equation}
    We now assume~\eqref{eq:StabLem:X} in order to derive a contradiction, and
    make a final subcase distinction. 

    First assume that $\sqp(n,\delta+\eta
    n)<\frac{11}{20}n$. Then $r'\ge 2$ and hence~\eqref{eq:StabLem:X}
    and~\eqref{eq:sqpc} imply 
    \[
      |X|<\tfrac12(\delta+\eta n)+3
      <\tfrac{51}{100}\delta
      <\tfrac{19}{10}(2\delta-n)\,,
    \] 
    because
    $\delta\ge(\frac35-2\eta)n$ and $\eta\le\frac{1}{1000}$. Furthermore,
    since~$G$ has two triangle components whose exterior is of size at most~$X$
    by assumption
    we have 
    $|\intr(G)|>n-2|X|=n-\delta-\eta n-6$, a contradiction to the
    the conditions of Fact~\ref{fac:StabLem:around35}. 

    Now assume that $\sqp(n,\delta+\eta n)\ge\frac{11}{20}n$. Then we have
    $\delta>(\frac23-2\eta)n$. By Lemma~\ref{lem:cpts}\ref{lem:cpts:a}
    we have $|X|\le n-\delta<(\frac13+2\eta)n$ and so
    $|X|<\frac{19}{10}(2\delta-n)$. Further $|\intr(G)|\ge n-2|X|\ge 2\delta-n>\frac{n}{3}-4\eta n>n-\delta-11\eta
    n$, which again contradicts the conditions of Fact~\ref{fac:StabLem:around35}.
  \end{factproof}
  
  \begin{fact}\label{fac:StabLem:35} 
    Lemma~\ref{StabLem} is true for $\frac{3n-2}{5}\leq\delta<\frac{2n-1}{3}$.  
  \end{fact} 
  \begin{factproof}[Proof of Fact~\ref{fac:StabLem:35}]
    Observe that in this range $r=2$. Assume $G$ has an edge $uv$ in $\intr(G)$,
    let $x$ be a common neighbour of $u$ and $v$ and $C$ be the triangle
    component containing $ux$ and $vx$. Since $uv\in\intr(G)$ there are edges
    $uy$ and $vz$ of $G$ outside $C$. The sets $\Gamma(u,y)$, $\Gamma(v,z)$ and
    $\{u,v,x,y,z\}$ are pairwise disjoint, and $x$ is not adjacent to
    $\Gamma(u,y)\cup\Gamma(v,z)\cup\{y,z\}$. So $\delta\le
    d(x)\le(n-1)-2(2\delta-n)-2$ which is only possible when
    $\delta\le(3n-3)/5$, a contradiction. Thus $\intr(G)$ is an independent set,
    which implies $|\intr(G)|\le n-\delta$. Hence, by
    Lemma~\ref{lem:cpts}\ref{lem:cpts:a}, $G$ cannot have more than two triangle
    components. In particular, all vertices in $\intr(G)$ lie in both triangle
    components of $G$. So if $|\intr(G)|\ge n-\delta-11\eta n$ then $\intr(G)$
    is the desired large independent set for Case~\ref{StabLem:3}. If moreover
    all triangle component exteriors are of size $\frac{19}{10}(2\delta-n)$ at
    most we are in Case~\ref{StabLem:3}. Otherwise (if $\intr(G)$ is small or a
    triangle component exterior is large) Fact~\ref{fac:StabLem:around35} gives
    $\CTF(G)\ge\min(\sqp(n,\delta+\eta n),\frac{11}{20}n)$ which is
    Case~\ref{StabLem:2}.
  \end{factproof}
   
  For the remainder of the proof, we suppose $\delta<\frac{3n-2}{5}$ and
  accordingly $r\ge 3$ and $r'\ge 2$. For dealing with this case we first
  establish two auxiliary facts. The first one captures the greedy technique
  for finding a large connected triangle factor that we sketched in the
  beginning of this section.  We will use this technique throughout the rest of
  the proof.
   
  \begin{fact} 
  \label{fac:StabLem:greedy} 
    If there are two sets $U_1,U_2\subset V(G)$ 
    such that no vertex in $U_1$ has a neighbour in $U_2$, all edges in 
    $G[U_1]$ are triangle connected and $\delta(G[U_1])\ge\delta_1$ then 
    $\CTF(G)\ge\min(3\lfloor|U_1|/2\rfloor,3\delta_1,2\delta-n+|U_2|)$. 
  \end{fact} 
  \begin{factproof}[Proof of Fact~\ref{fac:StabLem:greedy}]
    By Proposition~\ref{prop:match}\ref{prop:match:a} we can find a matching $M'$
    in~$U_1$ covering
    \[\min(2\lfloor|U_1|/2\rfloor,2\delta_1)\] 
    vertices. Let $M$ be a subset of $M'$ covering
    $\min(2\lfloor|U_1|/2\rfloor,2\delta_1,(4\delta-2n+2|U_2|)/3)$ vertices. For
    each edge $e\in M$ in turn we pick greedily a common neighbour of $e$
    outside both $M$ and the previously chosen common neighbours to obtain a set
    $T$ of disjoint triangles.  For any $x,y\in U_1$ we have $|\Gamma(x,y)|\geq
    2\delta-(n-|U_2|)$.  
    We claim that this implies that~$T$ can be constructed, covering all
    of~$M$. Indeed, in each step of the greedy procedure we have strictly more
    than $2\delta-(n-|U_2|)-3|M|\ge 0$ common neighbours of $e\in M$ available. 
    Hence $T$ covers at least
    $\min(3\lfloor|U_1|/2\rfloor,3\delta_1,2\delta-n+|U_2|)$ vertices.  Note
    further that~$T$ is a connected triangle factor because all edges in
    $G[U_1]$ are triangle connected.
  \end{factproof}
   
  Below, our goal will be to show that $\intr(G)$ is an independent set. The
  following fact prepares us for this step. 

  \begin{fact} 
  \label{fac:StabLem:A} 
    Let $uv$ be an edge in $\intr(G)$.  
    Unless $r'=2$ at least one vertex, $u$ or $v$, is contained in at most 
    $r'-1$ triangle components. 
  \end{fact} 
   \begin{factproof}[Proof of Fact~\ref{fac:StabLem:A}]
     Let $C_1$ be the triangle component containing $uv\in\intr(G)$ along with
     the (non-empty) common neighbourhood $\Gamma(u,v)$ (and perhaps some other
     neighbours of $u$ or $v$ separately). Suppose that $C\neq C_1$, and $u$ is
     a vertex of $C$. Then by Lemma~\ref{lem:cpts}\ref{lem:cpts:d}, there are at
     least $2\delta-n+1$ neighbours $x$ of $u$ such that the edge $ux$ is in
     $C$. Now suppose that $u$ lies in at least $r'-1$ triangle components other
     than $C_1$. It follows that there is a set $U_u\subset\Gamma(u)$ of
     vertices $x$ such that $ux$ is not in $C_1$, with
     $|U_u|\geq(r'-1)(2\delta-n+1)$, since no edge lies in two distinct triangle
     components. Suppose furthermore that $v$ too lies in at least $r'-1$
     triangle components other than $C_1$. Then there exists an analogously defined
     set $U_v$. Since all vertices of $\Gamma(u,v)$ form triangles of $C_1$ with
     $u$ and $v$, the three sets $\Gamma(u,v)$, $U_u$ and $U_v$ are pairwise
     disjoint, and thus $|U_u\cup U_v|\geq (2r'-2)(2\delta-n+1)$.  Now given any
     $x\in\Gamma(u,v)$, since $ux$ and $vx$ are both in $C_1$, $x$ cannot be
     adjacent to any vertex of $U_u\cup U_v$. But then $\delta\le
     d(x)<n-(2r'-2)(2\delta-n+1)$ which is equivalent to
     $2r'-2<(n-\delta)/(2\delta-n+1)$. By~\eqref{eq:r} the right-hand side is at
     most $r$ and thus we get $2r'-2<r$. Since $r\le r'+1$ however this is a
     contradiction unless $r'\le 2$. 
   \end{factproof}

  We assume from now on, that
  \begin{equation}\label{eq:StabLem:co}
   \CTF(G)<\sqp(n,\delta+\eta n)\,,
  \end{equation}
  that is, we are not in Cases~\ref{StabLem:1} or~\ref{StabLem:2}.
  Our aim is to conclude that then ($*$) $\intr(G)$ is an independent set and 
  that its vertices are contained in at least $r'$ triangle components. It
  turns out, however, that we need to consider the cases $r=r'+1=3$ and
  $r=r'+1=4$ (i.e., the cases when the minimum degree $\delta$ is just a
  little bit below $\frac35n$ and~$\frac47n$, respectively) separately.
  Unfortunately these two cases, which are treated by
  Fact~\ref{fac:StabLem:int:special}, require a somewhat technical case
  analysis, which we prefer to defer to the end of the section.
  
  \begin{fact} 
  \label{fac:StabLem:int:special} 
    If $r=r'+1=3$ or $r=r'+1=4$ then $\intr(G)$ is an independent set  
    all of whose vertices are contained in at least $r'$ triangle components. 
  \end{fact} 
   
  Assuming this fact is true we can deduce ($*$) for all values $r\ge 3$ as 
  follows. 
   
  \begin{fact} 
  \label{fac:StabLem:int} 
    The set $\intr(G)$ is an independent set  
    (and hence of size at most $n-\delta$) all of whose vertices are contained 
    in at least $r'$ triangle components. 
  \end{fact} 
  \begin{factproof}[Proof of Fact~\ref{fac:StabLem:int}]
    Recall that we have $r\ge 3$ at this point of the proof. Moreover,
    the cases $r=r'+1=3$ and $r=r'+1=4$ are handled by
    Fact~\ref{fac:StabLem:int:special}. So we assume we are not in these cases; in
    particular, $r'\geq 3$. We will show that then each vertex of $\intr(G)$ is
    contained in at least~$r'$ triangle components. Once we establish this,
    Fact~\ref{fac:StabLem:A} implies that there are no edges in $\intr(G)$ and
    so $\intr(G)$ is an independent set as desired.
 
    To prove that each vertex of $\intr(G)$ is contained in at least $r'$
    triangle components we assume the contrary and show that then
    $\CTF(G)\ge\sqp(n,\delta+\eta n)$, a contradiction
    to~\eqref{eq:StabLem:co}. Indeed, let $w\in\intr(G)$ and suppose that there
    are $k>1$ triangle components $C_1,\dots,C_k$ containing~$w$. For $i\in[k]$
    let~$U_i$ be the set of neighbours~$u$ of~$w$ such that $uw\in C_i$.
    By
    Lemma~\ref{lem:cpts}\ref{lem:cpts:d} we have $\delta(G[U_i])\geq 2\delta-n$
    and $|U_i|\ge2\delta-n+1$. Suppose that $U_1$ is the largest of the $U_i$.
    No vertex in $U_1$ has a neighbour in $U_2$, since the components are
    distinct. In addition, all edges in $G[U_1]$ are triangle connected, because
    $U_1\subset \Gamma(w)$. Therefore Fact~\ref{fac:StabLem:greedy} implies
    that there is a connected triangle factor $T$ in $G$ covering
    $\min(3\lfloor|U_1|/2\rfloor,3(2\delta-n),2\delta-n+|U_2|)\ge
    \min(3\lfloor|U_1|/2\rfloor,4\delta-2n)$ vertices.
    If $w$ lies only in $r'-1$ triangle components then $|U_1|\ge\delta/(r'-1)$
    and therefore $T$ covers at least
    $\min(3\lfloor\delta/(2r'-2)\rfloor,4\delta-2n)$ vertices.  Now
    since~\eqref{eq:sqpa} holds, we have
    $\frac32\delta/(r'-1)-2\ge\sqp(n,\delta+\eta n)$. Since $r\geq r'\geq 3$ and
    we have excluded the case $r=r'+1=4$, by~\eqref{eq:sqpb} we have
    $4\delta-2n\ge\sqp(n,\delta+\eta n)$. It follows that $T$ covers at least
    $\sqp(n,\delta+\eta n)$ vertices, in contradiction
    to~\eqref{eq:StabLem:co}.
  \end{factproof}
 
  \begin{fact}\label{fac:StabLem:case3} 
    We are in Case~\ref{StabLem:3}. 
  \end{fact} 
  \begin{factproof}[Proof of Fact~\ref{fac:StabLem:case3}]
    Fact~\ref{fac:StabLem:int} tells us that $\intr(G)$ is an independent
    set. By Lemma~\ref{lem:cpts}\ref{lem:cpts:a} and the fact that
    $\delta>n-\delta$ we have that every triangle component in $G$ has an
    exterior, and by Lemma~\ref{lem:cpts}\ref{lem:cpts:c} that there are no
    edges between any triangle component exteriors. Hence, to show that
    we are in Case~\ref{StabLem:3}, it is enough to prove that
    \begin{equation}\label{eq:Fact7:goal}
      |\intr(G)|:=\alpha\ge n-\delta-11\eta n\qquad\text{and}\qquad
      |X_1|\le\frac{19}{10}(2\delta-n)
    \end{equation}
    for the biggest triangle component exterior $X_1$ in $G$.  Suppose for a
    contradiction that this is not the case. We first claim that this forces $G$ to
    have exactly $r'$ triangle components.
  
    Indeed, assume $G$ has $k\ge r'+1$ triangle components. Each of these
    components $C$ has vertices in its exterior $\bound(C)$, and so by
    Lemma~\ref{lem:cpts}\ref{lem:cpts:c} the minimum degree of $G$ implies
    $|\bound(C)|\ge\delta-\alpha+1\ge2\delta-n+1$. We let these triangle
    component exteriors be $X_1,\ldots,X_k$, with $X_1$ being the biggest. Since
    $n=|X_1\dcup\ldots\dcup X_k\dcup\intr(G)|$, we have
    $(r'+1)(\delta-\alpha)+\alpha< n$. We distinguish two cases.

    {\sl Case 1:} \eqref{eq:Fact7:goal} fails because $\alpha<n-\delta-11\eta n$. 
    Then we obtain
    \begin{equation*}
      \begin{split}
        (r'+1)\delta &< n+r'\alpha<n+r'(n-\delta-11\eta n) \\
        & =(r'+1)n-(9r'-1)\eta n-r'\delta-(2r'+1)\eta n\,.
      \end{split}
    \end{equation*}
    Straightforward manipulation gives
    \[\delta+\eta n<\frac{(r'+1)n-(9r'-1)\eta n}{2(r'+1)-1}\,.\]
    Since $(9r'-1)\eta n\ge 9r'-1\ge r'$ this contradicts~\eqref{eq:r} applied
    to $r'=\rp(n,\delta+\eta n)$. 

    {\sl Case 2:} \eqref{eq:Fact7:goal} fails because
    $|X_1|>\frac{19}{10}(2\delta-n)$.
    Let $x$ be any vertex in $X_2$. Since $x$ has at least
    $\delta$ neighbours, none of which are in $X_1\dcup X_3\dcup\ldots\dcup
    X_k$, we have
    \begin{align*}
      1+\delta+\frac{19}{10}(2\delta-n)+(k-2)(2\delta-n+1)&\le
      n\,\text{, hence}\\
      \frac{19}{10}(2\delta-n)+(r'-1)(2\delta-n)&<n-\delta\,.
    \end{align*}
    By~\eqref{eq:r} we have $r'\ge(n-\delta-\eta n)/(2\delta+2\eta
    n-n+1)$. Combined with the last inequality, this gives
    \begin{equation*} 
      \frac9{10}(2\delta-n) 
      +\frac{n-\delta-\eta n}{2\delta-n+1+2\eta n}(2\delta-n)< n-\delta 
    \end{equation*}
    Now provided that $\eta<2\mu^2/3$, and since $2\delta-n\ge2\mu n$, we have
    \begin{equation*}
      \begin{split}
        (2\delta-n+2\eta n+1)(1-\mu)
        & <2\delta-n+3\eta n-\mu(2\delta-n) \\
        & \leq 2\delta-n+3\eta n-2\mu^2 n <2\delta-n\,,
      \end{split}
    \end{equation*}
    and we obtain $\frac{9}{5}\mu n+(1-\mu)(n-\delta-\eta n)<n-\delta$ which is
    a contradiction since $n-\delta<n/2$ and $\eta<\mu$. 
 
    Hence,
    if~\eqref{eq:Fact7:goal} fails, then~$G$ has indeed exactly $r'$ triangle components.
 
    \smallskip

    Now we use this fact in order to derive a contradiction to~\eqref{eq:StabLem:co}.
    Observe that, if $r'=2$, and accordingly $\delta\ge(\frac35-2\eta)n$, then
    Fact~\ref{fac:StabLem:around35} implies that~\eqref{eq:Fact7:goal} holds,
    because according to~\eqref{eq:StabLem:co} we have $\CTF(G)<\sqp(n,\delta+\eta
    n)$. In the remainder we assume $r'\ge3$.
  
    Since every vertex in $X_1$ has neighbours only in $X_1$ and $\intr(G)$, and
    $|\intr(G)|\leq n-\delta$, we have $\delta(G[X_1])\geq
    2\delta-n$. Furthermore, since no vertex in $X_1$ has neighbours in either
    $X_2$ or $X_3$, and $|X_2\dcup X_3|\geq 2(2\delta-n+1)$, we can apply
    Fact~\ref{fac:StabLem:greedy} to obtain
    \begin{align*}
      \CTF(G)&\geq\min\big(3\lfloor|X_1|/2\rfloor, 3(2\delta-n),
      2\delta-n+2(2\delta-n+1)\big)\\
      &= \min\big(3\lfloor|X_1|/2\rfloor,
      3(2\delta-n)\big)\,.
    \end{align*}
    Now by~\eqref{eq:sqpb}, $\CTF(G)\geq 3(2\delta-n)$ is a contradiction
    to~\eqref{eq:StabLem:co}, so to complete our proof it remains to show that
    if~\eqref{eq:Fact7:goal} fails, then $\CTF(G)\geq 3\lfloor|X_1|/2\rfloor$ is
    also a contradiction to~\eqref{eq:StabLem:co}. Again, we distinguish two
    cases.

    {\sl Case 1:} \eqref{eq:Fact7:goal} fails because $\alpha<n-\delta-11\eta n$.
    Since $X_1$ is the largest exterior,
    we have $|X_1|\geq (\delta+11\eta n)/r'$. But we have by~\eqref{eq:sqpa}
    that
    \[\sqp(n,\delta+\eta n)\le\frac32\frac{\delta+3\eta
      n}{r'}-2<3\Big\lfloor\frac{\delta+11\eta n}{2r'}\Big\rfloor\,,\] so that
    $\CTF(G)\geq 3\lfloor|X_1|/2\rfloor$ is indeed a contradiction
    to~\eqref{eq:StabLem:co}.

    {\sl Case 2:} \eqref{eq:Fact7:goal} fails because
    $|X_1|>\frac{19}{10}(2\delta-n)$.  Then $\CTF(G)\geq
    3\lfloor|X_1|/2\rfloor\geq\frac{57}{20}(2\delta-n)-2$, which
    by~\eqref{eq:sqpb} is a contradiction to~\eqref{eq:StabLem:co}, as desired.
  \end{factproof}
 This completes, modulo the proof of Fact~\ref{fac:StabLem:int:special}, the
 proof of Lemma~\ref{StabLem}.
\end{proof} 
 
It remains to show Fact~\ref{fac:StabLem:int:special}. Note that we can use all
facts from the proof of Lemma~\ref{StabLem} that precede
Fact~\ref{fac:StabLem:int:special}. We will further assume that all constants and
variables are set up as in this proof.
 
\begin{proof}[Proof of Fact~\ref{fac:StabLem:int:special}]
  Recall that we assumed~\eqref{eq:StabLem:co}, i.e., $\CTF(G)<\sqp(n,\delta+\eta
  n)$, in this part of the proof of Lemma~\ref{StabLem}. We distinguish two cases.

  \smallskip

  {\sl Case 1:} $r=3$ and $r'=2$. In this case
  $\delta(G)\in[(\frac35-2\eta)n,(\frac35+\eta)n]$. Trivially each vertex of
  $\intr(G)$ is contained in at least $r'=2$ triangle components. Suppose for a
  contradiction that there is an edge $uv$ in $\intr(G)$. Let $x$ be a common
  neighbour of $u$ and $v$, and $C$ be the triangle component containing the
  triangle $uvx$.
  Let $U_1:=\{y:uy\in C\}$ and $V_1:=\{y:vy\in C\}$ and let
  $U_2:=\Gamma(u)\setminus U_1$ and $V_2:=\Gamma(v)\setminus V_1$. Observe that
  $U_2\cap V_2=\emptyset$.
  
  By definition $x$ is not in, and has no neighbour in, $U_2\dcup V_2$. It
  follows that $|U_2\dcup V_2|<n-\delta\leq (\frac{2}{5}+2\eta)n$. On the other
  hand, by Lemma~\ref{lem:cpts}\ref{lem:cpts:d}, we have $|U_2|$,
  $|V_2|>2\delta-n\ge\frac15n-4\eta n$, and thus
  \[|U_2|,|V_2|\in\big[(\tfrac{1}{5}-4\eta)n,(\tfrac{1}{5}+6\eta)n\big]\,.\]
  Since $d(u)\geq\delta\geq(\tfrac{3}{5}-2\eta)n$, we have
  $|U_1|\geq\delta-|U_2|\geq(\tfrac{2}{5}-8\eta)n$. But no vertex in $U_2$ is
  adjacent to any vertex in $U_1$. This implies that every vertex in $U_2$ is
  adjacent to all but at most $n-\delta-|U_1|\leq 10\eta n$ vertices outside
  $U_1$. Since $\eta<\tfrac{1}{1000}$ we have $|U_2|>20\eta n$, so
  $\delta(G[U_2])>|U_2|/2$, and by
  Proposition~\ref{prop:match}\ref{prop:match:a}, $U_2$ contains a matching $M_u$
  with $\lfloor|U_2|/2\rfloor$ edges. Since each vertex of $U_2$ has at most
  $10\eta n$ non-neighbours outside $U_1$, each pair of vertices has common
  neighbourhood covering all but at most~$20\eta n$ vertices of $V(G)\setminus
  U_1$. In particular, the common neighbourhood of each edge of $M_u$ covers all
  but at most~$20\eta n$ vertices of $V(G)\setminus U_1$. Similarly, $V_2$
  contains a matching $M_v$ with $\lfloor|V_2|/2\rfloor$ edges, and the common
  neighbourhood of each edge covers all but at most~$20\eta n$ vertices of
  $V(G)\setminus V_1$.
  
  Since $20\eta n<|U_2|/4$ and $U_2\cap V_1=\emptyset$, the common neighbourhood
  of each edge of $M_v$ contains more than half of the edges of $M_u$. By symmetry, the reverse is also
  true. Thus all edges in $M_u\dcup M_v$ are in the same triangle component of
  $G$. Finally, each edge of $M_u\dcup M_v$ has at least $\delta-10\eta
  n-|U_2\dcup V_2|\geq (\tfrac{1}{5}-24\eta)n$ common neighbours outside
  $U_2\dcup V_2$. Choosing greedily for each edge of $M_u\dcup M_v$ in succession
  distinct common neighbours outside $U_2\dcup V_2$, we obtain a connected
  triangle factor with $\min(\lfloor|U_2|/2\rfloor+\lfloor|V_2|/2\rfloor,
  (\tfrac{1}{5}-24\eta)n)=(\tfrac{1}{5}-24\eta)n$ triangles. But then
  $\CTF(G)\geq(\tfrac{3}{5}-72\eta)n>n/2>\sqp(n,\delta+\eta n)$, a contradiction
  to~\eqref{eq:StabLem:co}. This proves Fact~\ref{fac:StabLem:int:special} for
  the case $r=3$ and $r'=2$.

  
  \smallskip
   
  {\sl Case 2:} $r=4$ and $r'=3$. This implies that
  $(\frac47-2\eta)n\le\delta(G)\le(\frac47+\eta)n$, and consequently
  $\sqp(n,\delta+\eta n)<(\frac27+2\eta)n$. We first prove two statements
  about the structure of $G$ which are forced by~\eqref{eq:StabLem:co}.
  
  \begin{enumerate}[label={\rm($\Psi$)}]
      \item\label{StabLem:47:a}
       If a vertex~$u$ has sets of neighbours $U$, $U'$ on edges in exactly two different
       triangle components with $|U|\ge|U'|$ then $(\frac17-4\eta)
       n<|U'|<(\frac17+6\eta)n$ and $(\frac37-8\eta)n<|U|<(\frac37+2\eta)n$.
  \end{enumerate}
  \begin{factproof}[Proof of~\ref{StabLem:47:a}]
    For the lower bound on $|U'|$, observe that by~\ref{lem:cpts:d} of
    Lemma~\ref{lem:cpts} we have $\delta(G[U'])\geq
    2\delta-n\geq(\frac17-4\eta)n$. To obtain the upper bound, again by
    Lemma~\ref{lem:cpts}\ref{lem:cpts:d} we have $\delta(G[U])\geq 2\delta-n$,
    and since the sets $U$ and $U'$ are neighbours of $u$ in different triangle
    components $C$ and $C'$, there are no edges from $U$ to $U'$.  Furthermore,
    since any edge in $G[U]$ forms a triangle with $u$ using an edge from $u$ to
    $U$, all edges in $G[U]$ are in $C$. Now by Fact~\ref{fac:StabLem:greedy} we
    have
    \[\CTF(G)\geq\min(3\lfloor|U|/2\rfloor,3(2\delta-n),2\delta-n+|U'|)\,.\]
    Since $|U|\geq\delta/2$ we have
    $3\lfloor|U|/2\rfloor\geq(\tfrac{3}{7}-3\eta)n-2>\sqp(n,\delta+\eta n)$.
    By~\eqref{eq:sqpb} we have $3(2\delta-n)>\sqp(n,\delta+\eta
    n)$. Because~\eqref{eq:StabLem:co} holds, we have
    $2\delta-n+|U'|<\sqp(n,\delta+\eta n)<(\frac27+2\eta)n$, and therefore
    $|U'|<(\frac17+6\eta)n$.
    Now the claimed lower and upper bounds on $|U|$ follow from
    $U=\Gamma(u)\setminus U'$, and from the fact that no vertex in $U'$ has a
    neighbour in $U$, respectively.
  \end{factproof}
  
 \begin{enumerate}[label={\rm($\Xi$)}]
     \item\label{StabLem:47:b}
       If a vertex~$u$ has sets of neighbours $U_1$, $U_2$, $U_3$ on edges in exactly
       three different triangle components then
       $(\frac4{21}+2\eta)n>|U_i|>(\frac4{21}-6\eta) n$ for $i\in[3]$.
 \end{enumerate}
 \begin{factproof}[Proof of~\ref{StabLem:47:b}] 
   Assume that $U_1$ is the largest of the three sets. By~\ref{lem:cpts:d} of
   Lemma~\ref{lem:cpts} we have $\delta(G[U_i])\geq
   2\delta-n\geq(\frac17-4\eta)n$ for each $i$, so $|U_i|>(\frac17-4\eta)n$ for
   each $i$. As in the previous case, there can be no edge from $U_1$ to
   $U_2\dcup U_3$, and all edges in $U_1$ are triangle-connected. Thus by
   Fact~\ref{fac:StabLem:greedy} we have
  \[\CTF(G)\geq\min\big(3\lfloor|U_1|/2\rfloor,3(2\delta-n),2\delta-n+|U_2\dcup
  U_3|\big)\,.\] 
  Now since $\sqp(n,\delta+\eta n)<(\tfrac{3}{7}-10\eta)n$
  and~\eqref{eq:StabLem:co} holds, we have
  \[3\lfloor|U_1|/2\rfloor<\sqp(n,\delta+\eta n)\leq (\frac27+2\eta)n\,\]
  which implies $|U_1|<(\frac4{21}+2\eta)n$. Since $|U_2|,|U_3|\leq |U_1|$ this
  completes the desired upper bounds. The lower bounds follow from
  $|U_1|+|U_2|+|U_3|\geq\delta\geq(\tfrac{4}{7}-2\eta)n$.
 \end{factproof}
  
  Next we show that 
 \begin{enumerate}[label={\rm($\Theta$)}]
     \item\label{StabLem:47:c}
       $\intr(G)$ is an independent set. 
 \end{enumerate}
 \begin{factproof}[Proof of~\ref{StabLem:47:c}] 
   Assume for a contradiction that there is an edge $uv\in\intr(G)$. By
   Fact~\ref{fac:StabLem:A} one of the vertices of this edge, say $u$, is in
   only~$2$ triangle components. Let its neighbours be $U_1$ and $U_2$ in these
   two triangle components, and let the neighbours of $v$ be partitioned into
   sets $V_1,\dots,V_k$ according to the triangle component containing the edge
   to $v$. Assume further that $\Gamma(u,v)\subset U_1\cap V_1$, so that
   $U_2,V_2,\ldots,V_k$ are pairwise disjoint. Let $x\in\Gamma(u,v)$. Since $x$
   has neighbours in neither $U_2$ nor $V_2$, and since by
   Lemma~\ref{lem:cpts}\ref{lem:cpts:d} we have $|V_2|>(\frac17-4\eta)n$, we
   conclude that $\delta\le d(x)\leq n-1-|U_2|-|V_2|$. In particular,
   $|U_2|<(\frac37-8\eta)n$ because $\delta\ge(\frac47-2\eta)n$, and therefore
   by~\ref{StabLem:47:a} we have   
   \[(\frac17-4\eta) n<|U_2|<(\frac17+6\eta)n\,.\]
   Next we want to derive analogous bounds for~$|V_2|$. For this purpose we first
   show that~$k=2$.

   Indeed, if we had $k=3$, then by~\ref{StabLem:47:b}
   \begin{equation*}
     \begin{split}
       d(x) & \leq n-1-|U_2|-|V_2|-|V_3| \\
       &\leq n-1-(\tfrac{1}{7}-4\eta)n-2(\tfrac{4}{21}-6\eta)n 
       <(\tfrac{10}{21}+16\eta)n<\delta\,,
     \end{split}
   \end{equation*}
   and this contradicts $\delta(G)\geq\delta$. Similarly, if $k\geq 4$, then by
   Lemma~\ref{lem:cpts}\ref{lem:cpts:d} we have $|V_i|\geq(\tfrac{1}{7}-4\eta)n$
   for each $i$, and hence
   \[d(x)\leq n-1-|U_2|-|V_2|-|V_3|-|V_4|<(\tfrac{3}{7}+16\eta)n<\delta\,,\]
   which too is a contradiction. It follows that $k=2$ as claimed. 

   Hence, we can argue analogously as before (for~$U_2$) that
   $|V_2|>(\frac37-8\eta)$ would contradict $d(x)\geq\delta$. Consequently,
   by~\ref{StabLem:47:a} we have
   \[(\frac17-4\eta) n<|V_2|<(\frac17+6\eta)n\,.\]
  
   We now argue that this yields a contradiction to~\eqref{eq:StabLem:co} in
   much the same way as we argued in the $r=r'+1=3$ case. Every vertex of $U_2$
   is adjacent to all but at most $n-|U_1|-\delta\le 10\eta n$ vertices of
   $V(G)\setminus U_1$. Since $|U_2|>20\eta n$, by
   Proposition~\ref{prop:match}\ref{prop:match:a} there is a matching $M_u$ in
   $U_2$ covering all but at most one vertex of $U_2$. Each edge of $M_u$ has at
   least $\delta-10\eta n\geq(\tfrac{4}{7}-12\eta)n$ common neighbours outside
   $U_1$.  Similarly, in $V_2$ there is a matching $M_v$ covering all but at
   most one vertex of $V_2$, each edge of which has at least
   $(\tfrac{4}{7}-12\eta)n$ common neighbours outside $V_1$. Since
   $\Gamma(u,v)=U_1\cap V_1$, we have $U_1\cap V_2=\emptyset$. It follows that
   every edge of~$M_v$ has more than half of the edges of $M_u$ in its common
   neighbourhood, and thus the edges $M_u\dcup M_v$ are triangle
   connected. Choosing greedily for each edge in $M_u\dcup M_v$ in succession a
   distinct common neighbour outside $M_u\dcup M_v$, we obtain a connected
   triangle factor with as many triangles as there are edges in $M_u\dcup
   M_v$. Since $|U_2|,|V_2|>(\frac17-4\eta) n$, we have
   $\CTF(G)>(\tfrac{3}{7}-12\eta)n-3>\sqp(n,\delta+\eta n)$,
   contradicting~\eqref{eq:StabLem:co}. This completes the proof that $\intr(G)$
   is an independent set.
 \end{factproof}
 
  It remains to show that each vertex $u\in\intr(G)$ is contained in at least
  $r'=3$ triangle components. Assume for a contradiction that this is not the
  case and that some vertex $u$ is only contained in $2$ triangle components, $C$ and
  $C'$. Let $U$ and $U'$, respectively, be the neighbours of~$u$ on edges in $C$ and
  $C'$. Without loss of generality $|U|\ge|U'|$. Because $\intr(G)$ is an
  independent set, $U$ and $U'$ are contained in the exteriors of $C$ and $C'$.
  By Lemma~\ref{lem:cpts}\ref{lem:cpts:c} there are thus no edges between~$U$
  and~$\bound(C')$. By Lemma~\ref{lem:cpts}\ref{lem:cpts:d} we have
  $\delta(G[U])\geq 2\delta-n$, and since $U\subset\bound(C)$ every edge of
  $G[U]$ is in $C$. It follows that we may apply Fact~\ref{fac:StabLem:greedy}
  to obtain
 
  \[\CTF(G)\geq
  \min\big(3\lfloor|U|/2\rfloor,3(2\delta-n),2\delta-n+|\bound(C')|\big)\,.\]
  Since $|U|\geq\delta/2$ we have
  $3\lfloor|U|/2\rfloor\geq(\tfrac{3}{7}-3\eta)n-2>\sqp(n,\delta+\eta n)$.
  By~\eqref{eq:sqpb} we have $3(2\delta-n)>\sqp(n,\delta+\eta n)$.
  Since~\eqref{eq:StabLem:co} holds, we conclude that
  $2\delta-n+|\bound(C')|<\sqp(n,\delta+\eta n)<(\frac27+2\eta)n$, and
  therefore $|\bound(C')|<(\frac17+6\eta)n$.
  
  Now any vertex in $\bound(C')$ has neighbours only in
  $\bound(C')\dcup\intr(G)$, and therefore
  $|\intr(G)|\geq\delta-|\bound(C')|\geq (\tfrac{3}{7}-8\eta)n$. The vertex $u$
  has neighbours only in $U'\subset\bound(C')$ and $U$, and
  therefore
  \[|U|\geq\delta-|U'|\geq\delta-|\bound(C')|\geq(\tfrac{3}{7}-8\eta)n\,.\]
  By Lemma~\ref{lem:cpts}\ref{lem:cpts:d} we have $\delta(G[U])\geq
  2\delta-n\geq(\tfrac{1}{7}-4\eta)n$, and since $|U|>(\tfrac{2}{7}-8\eta)n$ we
  obtain by Proposition~\ref{prop:match}\ref{prop:match:a} a matching $M$ in
  $U$ with at least $(\tfrac{1}{7}-4\eta)n$ edges. Now each vertex in
  $\intr(G)$ is adjacent to all but at most $n-\delta-|\intr(G)|\leq 10\eta n$
  vertices outside $\intr(G)$. In particular, each vertex in $\intr(G)$ is
  adjacent to all but at most $10\eta n$ vertices of~$M$, and is therefore a
  common neighbour of all but at most $10\eta n$ edges of~$M$. We now match
  greedily vertices of $\intr(G)$ with distinct edges of~$M$ forming triangles.
  Since $|\intr(G)|>|M|$, we will be forced to halt only when we come to a vertex
  $x\in\intr(G)$ which is not a common neighbour of any remaining edge of $M$,
  i.e., when we have used all but at most $10\eta n$ edges of $M$. It follows
  that we obtain a triangle factor $T$ with at least $(\tfrac{1}{7}-14\eta)n$
  triangles. Since each triangle uses an edge of $M\subset G[U]\subset
  G[\bound(C)]$, $T$ is a connected triangle factor, and we have $\CTF(G)\geq
  (\tfrac{3}{7}-42\eta)n>\sqp(n,\delta+\eta n)$ in contradiction
  to~\eqref{eq:StabLem:co}.
\end{proof} 
 
\section{Near-extremal graphs} 
\label{sec:ext}

In this section we provide the proof of Lemma~\ref{lem:ext-like}. To prepare this
proof we start with two useful lemmas. The first will be used to construct
squared paths and squared cycles from simple paths and cycles.
 
\begin{lemma}\label{lem:squaringpath} Given a graph $G$, let
  $T=(t_1,t_2,\ldots,t_{2l})$ be a path in $G$ and $W$ a set of vertices
  disjoint from $T$. Let $Q_1=(t_1,t_2)$,
  $Q_i=(t_{2i-3},t_{2i-2},t_{2i-1},t_{2i})$ for all $1<i\le l$, and
  $Q_{l+1}=(t_{2l-1},t_{2l})$. If there exists an ordering $\sigma$ of
  $[l+1]$ such that for each $i$ the vertices in $Q_{\sigma(i)}$ have at
  least $i$ common neighbours in $W$, then there is a squared path
  \[(q_1,t_1,t_2,q_2,t_3,t_4,q_3,\ldots,t_{2\ell},q_{\ell+1})\] in $G$,
  with $q_i\in W$ for each $i$, using every vertex of $T$.
 
  If $T$ is a cycle on $2l$ vertices we let instead
  $Q_1=(t_{2l-1},t_{2l},t_1,t_2)$,
  $Q_i=(t_{2i-3},t_{2i-2},t_{2i-1},t_{2i})$ for all $1<i\le l$, and
  $\sigma$ be an ordering on~$[l]$. Then, under the same conditions, we
  obtain a squared cycle~$C_{3l}^2$.
\end{lemma} 
\begin{proof}We need only ensure that for each $i$ one can choose~$q_i$
  such that~$q_i$ is a common neighbour of $Q_i$ and the $q_i$ are
  distinct. This is possible by choosing for each $i$ in succession
  $q_{\sigma(i)}$ to be any so far unused common neighbour of
  $Q_{\sigma(i)}$.
\end{proof}  
 
The second lemma is a variant on Dirac's theorem and permits us to
construct paths and cycles of desired lengths which keep some `bad'
vertices far apart.
 
\begin{lemma}\label{lem:nicepathcycle} Let $H$ be a graph on $h$ vertices and
  $B\subset V(H)$ be of size at most $h/100$. Suppose that every vertex in
  $B$ has at least $9|B|$ neighbours in $H$, and every vertex outside $B$
  has at least $h/2+9|B|+10$ neighbours in $H$. Then for any given $3\leq
  \ell\leq h$ we can find a cycle~$T_\ell$ of length $\ell$ in $H$ on which
  no four consecutive vertices contain more than one vertex of
  $B$. Furthermore, if $x$ and $y$ are any two vertices not in $B$ and
  $5\leq \ell\leq h$, we can find an $\ell$-vertex path $T_\ell$ whose
  endvertices are $x$ and $y$ on which no four consecutive vertices contain
  more than one vertex of $B\cup\{x,y\}$.
\end{lemma} 
\begin{proof} 
  If we seek a path in $H$ from $x$ to $y$ then we create a `dummy edge'
  between $x$ and $y$. If we seek a cycle, let $xy$ be any edge of $H-B$.
 
  First we construct a path $P$ in $H$ covering $B$ with the desired
  property. Let $B=\{b_1,b_2,\ldots,b_{|B|}\}$. For each $1\leq i\leq
  |B|-1$, choose a vertex $u_i\in H-B$ adjacent to $b_i$ and a vertex
  $v_i\in H-B$ adjacent to~$b_{i+1}$. Because both $u_i$ and $v_i$ have
  $h/2+9|B|+10$ neighbours in~$H$, they have at least $18|B|+20$ common
  neighbours. At most $3|B|$ of these are either in $B$ or amongst the
  chosen $u_j,v_j$, and so we can find a so far unused vertex $w_i$
  adjacent to $u_i$ and $v_i$. Since we require only $|B|-1$ vertices
  $w_1,\ldots,w_{|B|-1}$ we can pick the vertices greedily.
 
  We let $v_0$ be yet another vertex adjacent to $b_1$, and $u_{|B|}$
  adjacent to~$b_{|B|}$, and choose any further vertices
  $w_0,v_0,w_{|B|},u_{|B|}$ such that
  \[P=(x,y,u_0,w_0,v_0,b_1,u_1,w_1,v_1,b_2,\ldots,v_{|B|-1},b_{|B|},u_{|B|},w_{|B|},v_{|B|})\]
  is a path on $4|B|+5$ vertices.
 
  Now we let $P'$ be a path extending $P$ in $H$ of maximum length. We
  claim that $P'$ is in fact spanning. Suppose not: let $u$ be an
  end-vertex of $P'$ and $v$ a vertex not on $P'$. Since $P'$ is maximal
  every neighbour of~$u$ is on $P'$, so $v(P')>h/2+9|B|+10$. If there
  existed an edge $u'v'$ of $P'-P$ with $u'u$ and $v'v$ edges of $H$, with
  $v'$ closer to $u$ on $P'$ than~$u'$, then we would have a longer path
  extending $P$ in $H$. Counting the edges leaving $u$ and $v$ yields a
  contradiction.
 
  Finally we let $u$ and $v$ be the end-vertices of the spanning path
  $P'$. If~$uv$ is an edge of $H$, or if $u'v'$ is an edge of $P'-P$, with
  $u'$ nearer to~$u$ on~$P'$ than~$v'$, such that $uv'$ and $u'v$ are edges
  of $H$, then we obtain a cycle $T$ spanning $H$ and containing $P$ as a
  subpath. Again edge counting reveals that such an edge must exist.
 
  To obtain a cycle~$T_\ell$ with $h-|B|-2\leq \ell<h$ we take $u$ to be an
  end-vertex of the path $T-P$ and $v$ its successor on $T-P$. If we can
  find two further vertices $u'$ and $v'$ on $T-P$ (in that order from $u$
  along $T-P$) with $h-\ell$ vertices between them and with $uu'$ and $vv'$
  edges of $H$ then we would obtain a cycle $T_\ell$ of length
  $\ell$. Again simple edge counting reveals that such a pair of vertices
  exists. To obtain a cycle~$T_\ell$ with $3\leq \ell<h-|B|-2$ we note that
  $H-B$ has minimum degree $h/2+8|B|+10>(h-|B|)/2+1$ and thus contains a
  cycle of every possible length using the edge $xy$.
 
  The cycle $T_\ell$ satisfies the condition that no four consecutive
  vertices contain more than one vertex of $B$, since either it preserves
  $P$ as a subpath or it contains no vertices of $B$ at all. Similarly the
  path from~$x$ to $y$ within $T_\ell$ satisfies the required conditions.
\end{proof}  
 
Before embarking upon the proof of Lemma~\ref{lem:ext-like} we give an outline
of the method.
We recall that the Szemer\'edi partition supplied to the Lemma is
essentially the extremal structure. We shall show that the underlying graph
either also has an extremal structure, or possesses features which actually
lead to longer squared paths and cycles than required for the conclusion of
the Lemma. This is complicated by the fact that the Szemer\'edi partition
is insensitive both to mis-assignment of a sublinear number of vertices and
to editing of a subquadratic number of edges: we must assume, for example,
that although the vertex set $I$ in the reduced graph $R$ is independent,
the vertex set $\bigcup I$ may fail to contain some vertices of $G$ with no
neighbours in $\bigcup I$, and may contain a small number of edges meeting
every vertex. However, observe that by the definition of an
$(\eps,d)$-regular partition, there are no vertices of $\bigcup I$ with
more than $(\eps+d)n$ neighbours in $\bigcup I$.  Fortunately, it is
possible to apply the following strategy in this case.

We start by separating those vertices
with `few' neighbours in $\bigcup I$, which we shall collect in a set~$W$,
and those with `many'. We are then able to show (as
Fact~\ref{fac:2disjoint} below) that, if there are two vertex disjoint
edges in~$W$,
%
%
then the sets $\bigcup B_1$ and $\bigcup B_2$ are in the same triangle
component of~$G$ (`unexpectedly', since $B_1$ and $B_2$ are in different
triangle components in~$R$). We shall show that in this case it is possible
to construct very long squared paths and cycles by making use of
Lemma~\ref{lem:bl}.

Hence we can assume that there are not two disjoint edges in~$W$, which in turn 
implies that~$W$ is almost independent and will give us rather precise control
about the size of~$W$. In addition, the minimum degree condition will guarantee that
almost every edge from $W$ to the remainder of $G$ is present. We would like
to then say that in $V(G)\setminus W$ we can find a long path, which together
with vertices from $W$ forms a squared path (and similarly for squared
cycles). Unfortunately since $G[W,V(G)\setminus W]$ is not necessarily a
complete bipartite graph, this statement is not obviously true: although by
definition no vertex outside $W$ has very few neighbours in $W$, it is
certainly possible that two vertices outside $W$ could fail to have a common
neighbour in $W$. But the statement is true for a path possessing sufficiently
nice properties---specifically, satisfying the conditions of
Lemma~\ref{lem:squaringpath}---and the purpose of
Lemma~\ref{lem:nicepathcycle} is to provide paths and cycles with those nice
properties. The remainder of our proof, then, consists of setting up
conditions for the application of Lemma~\ref{lem:nicepathcycle}.
 
\begin{proof}[Proof of Lemma~\ref{lem:ext-like}] Given $\nu>0$, suppose the
  parameters $\eta>0$ and $d>0$ satisfy the following inequalities.
  \begin{equation}\label{ext:chooseetad}
    \eta\le \frac{\nu^4}{10^8} \qquad \text{and} \qquad d\le \frac{\nu^4}{10^8}
  \end{equation}
  Given $d>0$, Lemma~\ref{lem:bl} returns a constant $\eps_\subsc{EL}>0$. We set
  \begin{equation}\label{ext:chooseeps}
    \eps_0=\min\big(\frac{\nu^4}{10^8},\eps_\subsc{EL}\big) \,.
  \end{equation}
  Given $m_\subsc{EL}$ and $0<\eps<\eps_0$, Lemma~\ref{lem:bl} returns a constant
  $n_\subsc{EL}$. We set
  \begin{equation}\label{ext:chooseN}
    N=\max\big(1000m_\subsc{EL}^4,100\eta^{-1}\nu^{-1},n_\subsc{EL}\big) \,.
  \end{equation}
  Now let~$G$, $R$, and the partition $V(R)=I\dcup B_1\dcup\dots\dcup B_k$
  satisfy conditions \ref{ext:assm1}--\ref{ext:assm6} of the lemma.

  If $\delta(G)=\delta\geq \frac{2n-1}{3}$ then we can appeal to
  Theorem~\ref{thm:FanKierstead} to find a spanning squared path in $G$; if
  $\delta\geq \frac{2n}{3}$ then we can appeal to Theorem~\ref{thm:KSSPosa} to
  find $C^2_\ell$ for each $\ell\in [3,n]\setminus\{5\}$. Therefore, the
  definition of $\sqp(n,\delta)$ and $\sqc(n,\delta)$ imply that we may assume
  $\delta<2n/3$ in the following, and that we only need to find
  \begin{equation}\label{eq:ext:1120}
    \text{squared paths and squared cycles of length at most
      $11n/20$.}
  \end{equation}

  We now start by investigating the sizes of~$I$ and of the~$B_i$.
  Define $\delta' = (\delta/n-d-\eps)m$.
  Since~$R$ is an $(\eps,d)$-reduced graph we have
  \begin{equation}\label{eq:ext:delta'}
    \delta(R)\ge \delta' = (\delta/n-d-\eps)m\,.
  \end{equation}
  Observe that moreover
  \begin{equation}\label{eq:IisSmall}
    |I|\le m-\delta'\le \Big(1-\frac{\delta}{n}+d+\varepsilon\Big)m\;, 
  \end{equation} 
  by~\ref{ext:assm5}
  because clusters in $I$ have $\delta'$ neighbours outside $I$ in $R$. For
  $i\in[k]$, fix a cluster $C\in B_i$. By assumption~\ref{ext:assm6} $C$ has
  neighbours only in $B_i\cup I$ in
  $R$. Since \[\delta'\le\deg(C)=\deg(C,B_i\cup I)\le \deg(C,B_i)+|I|\le
  \deg(C,B_i)+m-\delta'\,,\] we have
  \begin{equation*}
    \begin{split}
      |B_i| & >\deg(C,B_i)\ge 2\delta'-m
      \ge\frac{m}{n}\big(2(\delta-dn-\eps n)-n\big) \\
      & =\frac{m}{n}\big(2\delta-n-(d+\eps)n\big)\,.
    \end{split}
  \end{equation*}
  Now since $2\delta-n\ge 2\nu n$ by~\ref{ext:assm1}, we conclude
  from~\eqref{ext:chooseetad} and~\eqref{ext:chooseeps} that
  \begin{equation}\label{eq:BiBig} 
    |B_i|\ge\frac{2m(2\delta-n)}{3n}\ge \frac43\nu m\;. 
  \end{equation} 

  We next show that each $B_i$ is part of exactly one triangle
  component of $R$.

  \begin{fact}\label{fac:BiTriConn} For each $1\le i\le k$ the following holds.
    All edges in $R[B_i]$ are triangle connected in $R$.
  \end{fact}
  \begin{factproof}[Proof of Fact~\ref{fac:BiTriConn}]
    By assumption~\ref{ext:assm6} we have
    \begin{equation}
      \label{eq:BiTriConn}
      |B_i|\le 19m(2\delta-n)/(10n)\le 39(2\delta'-m)/20\,,      
    \end{equation}
    where the second
    inequality comes from~\eqref{ext:chooseetad}
    and~\eqref{ext:chooseeps}. Since we have $\delta_R(B_i)\ge 2\delta'-m>|B_i|/2$,
    the graph $R[B_i]$ is connected. It follows that if there are two edges in
    $R[B_i]$ which are not triangle-connected, then there are two adjacent
    edges in $R[B_i]$ with this property. That is, there are vertices~$P$,~$Q$
    and~$Q'$ of $B_i$ such that $PQ$ is in triangle component~$C$ and $PQ'$
    is in triangle component~$C'$ with $C\neq C'$.
    
    We now show that there are at least $2\delta'-m$ edges leaving $P$ in
    $R[B_i]$ which are in $C$. There are two possibilities. First, suppose
    there are no $C$-edges from $P$ to $I$. In this case, the common
    neighbourhood $\Gamma(PQ)$ must lie entirely in $B_i$. Every vertex of
    $\Gamma(PQ)$ makes a $C$-edge with~$P$, and we have $|\Gamma(PQ)|\ge
    2\delta'-m$ as required. Second, suppose that there is a $C$-edge $PP'$
    with $P'\in I$. Since $I$ is an independent set in~$R$, the
    set~$\Gamma(PP')$ lies entirely within $B_i$, and has size at least $2\delta'-m$. Again, every edge
    from $P$ to $\Gamma(PP')$ is a $C$-edge, as desired.
    
    By the identical argument, there are at least $2\delta'-m$ edges leaving $P$
    in $R[B_i]$ which are in $C'$. Since no edge is in both $C$ and $C'$,
    there are at least $2(2\delta'-m)$ edges leaving $P$ in $R[B_i]$, so
    $|B_i|\ge 2(2\delta'-m)$. This contradicts~\eqref{eq:BiTriConn}.
    It follows that all edges of $B_i$ are triangle connected, as desired.
  \end{factproof}

  We next define a set~$W$ of those vertices in~$G$ which have few neighbours in
  $\bigcup I$. The intuition is that~$W$ consists of $\bigcup I$ and only a few
  more vertices of~$G$. 
  To simplify notation, we set
  $\xi=\sqrt[4]{\eps+d+11\eta}$. By~\eqref{ext:chooseetad}
  and~\eqref{ext:chooseeps}, we have
  \begin{equation}\label{ext:smallxi}
    \xi\le\nu/100\,.
  \end{equation}
  Let $W$ be the vertices of $G$ which do not have more than $\xi n$ neighbours
  in $\bigcup I$.  Since $\xi>d+\eps$, by the independence of $I$ and by the
  definition of an $(\eps,d)$-regular partition, we have $\bigcup I\subseteq W$.
  By assumption~\ref{ext:assm5} we have $|I|\ge(n-\delta-11\eta n)m/n$. Hence every
  edge in $W$ has at least
  \begin{equation}\label{eq:ext:Wedges}
    2(\delta-\xi n)-\Big(n-|\bigcup
    I|\Big)>\frac{\delta-(2\delta-n)}{16}
  \end{equation}
  common neighbours outside $\bigcup I$, where we use assumption~\ref{ext:assm1} that
  $2\delta-n> 2\nu n$,~\eqref{ext:chooseetad} and~\eqref{ext:smallxi}.

  By this observation, the next fact implies that we are done
  if there are two vertex disjoint edges in~$W$.

  \begin{fact}\label{fac:2disjoint} 
    If $u_1v_1$ and $u_2v_2$ are vertex disjoint edges of~$G$ such that for
    $i=1,2$ the edge
    $u_iv_i$ has at least $\delta-(2\delta-n)/16$ common neighbours outside
    $\bigcup I$, then $G$ contains $P^2_{\sqp(n,\delta)}$ and
    $C^2_\ell$ for each $\ell\in [3,\sqc(n,\delta)]\setminus\{5\}$.
  \end{fact} 
  \begin{factproof}[Proof of Fact~\ref{fac:2disjoint}]
    Let $D'$ be the set of clusters $C\in V(R)\setminus I$ such that~$u_1v_1$
    has at most $2dn/m$ common neighbours in $C$. By the hypothesis, $u_1v_1$
    has at least $\delta-(2\delta-n)/16$ common neighbours outside~$\bigcup I$.
    Of these, at most $\eps n$ are in the exceptional set $V_0$ of the regular
    partition, and at most $2dn|D'|/m$ are in $\bigcup D'$. The remaining
    common neighbours must all lie in $\bigcup(V(R)\setminus(I\cup D')$, and
    hence we have the inequality
    \begin{equation*}
      \begin{split}
        \delta-\frac{2\delta-n}{16}-\eps n-\frac{2dn|D'|}{m}
        & \leq (m-|I|-|D'|)\frac{n}{m} \\
        & \leBy{\ref{ext:assm5}} n-(n-\delta-11\eta n)-|D'|\frac{n}{m}\,.
      \end{split}
    \end{equation*}
    Simplifying this, we obtain
    \[\frac{n-2dn}{m}|D'|\leq 11\eta n+\eps n+\frac{2\delta-n}{16}\,, \]
    and by~\eqref{ext:chooseetad}
    and~\eqref{ext:chooseeps}, we get $|D'|\le (2\delta-n)m/(14n)$.

    Now let $D$ be the set of clusters $C\in V(R)\setminus I$ such that either
    $u_1v_1$ or~$u_2v_2$ has at most $2dn/m$ common neighbours in $C$. The same
    analysis holds for $u_2v_2$, so we
    obtain 
    \begin{equation}\label{eq:ext:D}
      |D|\le\frac{(2\delta-n)m}{7n}\,.
    \end{equation}
    Therefore, we conclude from~\eqref{eq:BiBig} that $B_1\setminus
    D\neq\emptyset$.  Take $X\in B_1\setminus D$ arbitrarily. We have
    \begin{equation*}
    \begin{split}
      \deg(X,B_1)
      & \geBy{\ref{ext:assm6}} \deg(X)-|I|
      \ge \delta'-|I|
      \geByRef{eq:IisSmall} \delta'-\Big(1-\frac{\delta}{n}+d+\eps\Big)m \\
      & \geByRef{eq:ext:delta'} \Big(\frac{\delta}{n}-d-\eps\Big)m
          -\Big(1-\frac{\delta}{n}+d+\eps\Big)m \\
      & \!\!\!\!\geBy{\eqref{ext:chooseetad},\eqref{ext:chooseeps}}
        \frac12(2\delta-n)\frac mn
      \gByRef{eq:ext:D} |D|
      \,.
    \end{split}
    \end{equation*}
    Thus there exists a cluster $Y\in \Gamma(X)\cap(B_1\setminus D)$.  Hence we
    have clusters $X,Y\in B_1\setminus D$ such that $XY\in E(R)$.  Analogously,
    we can find clusters $X',Y'\in B_2\setminus D$ such that $X'Y'\in E(R)$.
 
    Since $\delta_R(B_1),\delta_R(B_2)\geq \delta'-|I|\ge 2\delta'-m$, we can
    find greedily a matching~$M$ in $R[B_1\cup B_2]$ with $\delta'-|I|$
    edges. Since every cluster in~$I$ has at most $m-|I|-\delta'$ non-neighbours
    outside $I$, every cluster in~$I$ forms a triangle with at least
    $|M|-(m-|I|-\delta')=2\delta'-m$ edges of~$M$. 
    In addition, by assumption~\ref{ext:assm5}, \eqref{ext:chooseetad}, and since
    $\delta<2n/3$ we have $|I|>(\frac13-11\eta)m\ge\frac14 m$.
    Therefore we may choose greedily
    clusters in~$I$ to obtain a set~$T$ of at least 
    \begin{equation*}
      \min\big\{ 2\delta'-m, |I| \big\}
      \ge \min\Big\{ 2\delta'-m, \frac14 m \Big\}
    \end{equation*}
    vertex-disjoint triangles formed from edges of~$M$ and clusters of~$I$.
    Let~$T_1$ be the triangles of~$T$ contained in $B_1\cup I$, and~$T_2$ those
    contained in $B_2\cup I$.
    
    By Fact~\ref{fac:BiTriConn}, since each triangle in~$T_1$ contains an
    edge of $B_1$, all triangles in~$T_1$ are in the same triangle
    component as the edge $XY$.  Similarly all the triangles in~$T_2$ are
    in the same triangle component as the edge $X'Y'$.
      
    Noting that $\eps$ satisfies~\eqref{ext:chooseeps} and $n>N$
    satisfies~\eqref{ext:chooseN}, we can apply Lemma~\ref{lem:bl} with
    $X_1=X_2=X$, $Y_1=Y_2=Y$ to find a squared path starting with $u_1v_1$ and
    finishing with $u_2v_2$ using the triangles~$T_1$. Similarly, using
    Lemma~\ref{lem:bl} with $X_1=X_2=X'$, $Y_1=Y_2=Y'$ we find a squared path
    (intersecting the first only at $u_1$, $v_1$, $u_2$, and~$v_2$) starting
    with $u_2v_2$ and finishing with $u_1v_1$ using the triangles
    $T_2$. Choosing appropriate lengths for these squared paths and
    concatenating them we get a squared cycle $C^2_{\ell}$ in~$G$, for any
    $36(m_\subsc{EL}+2)^3\leq \ell\leq
    3(1-d)\min\{2\delta'-m,m/4\}n/m$. Applying Lemma~\ref{lem:bl} to the copy of
    $K_4$ in $B_1$ directly we obtain $C^2_\ell$ for each $\ell\in
    [3,3n/m]\setminus\{5\}$. By~\eqref{ext:chooseN} we have
    $3n/m>36(m_\subsc{EL}+2)^3$, and by~\eqref{eq:sqpb}, \eqref{ext:chooseetad},
    \eqref{ext:chooseeps}, and~\eqref{eq:ext:1120} we have
    $3(1-d)(2\delta'-m)n/m>\sqp(n,\delta)\ge\sqc(n,\delta)$ and $3(1-d)n/4\ge
    11n/20>\sqp(n,\delta)\ge\sqc(n,\delta)$. It follows that $G$ contains both
    $P^2_{\sqp(n,\delta)}$ and $C^2_\ell$ for each
    $\ell\in[3,\sqc(n,\delta)]\setminus\{5\}$ as required.
  \end{factproof}

  By~\eqref{eq:ext:Wedges}, if there are two vertex disjoint edges in $W$, then
  we are done by Fact~\ref{fac:2disjoint}. Thus we assume in the following that
  no such two edges exist.  This implies that there are two vertices in $W$
  which meet every edge in $W$. Since neither of these two vertices has more
  than $\xi n$ neighbours in $\bigcup I\subset W$, while $|I|>(\frac13-11\eta)m$
  by~\ref{ext:assm5} and because $\delta<2n/3$, there is a vertex in $W$
  adjacent to no vertex of $W$. We conclude that
  \begin{equation}
    \label{eq:ext-like:W}
    n-\delta-11\eta n\leq|\bigcup I|\leq
    |W|\leq n-\delta.
  \end{equation}

  Our next goal is to extract from each set $\bigcup B_i$ a large set $A_i$ of
  vertices which are adjacent to almost all vertices in~$W$ and are such that
  $G[A_i]$ has minimum degree somewhat above $|A_i|/2$.  Because at least
  $|W|\delta-2|W|$ edges leave $W$, the total number of non-edges between~$W$
  and $V(G)\setminus W$ is at most
  \[|W||V(G)\setminus W|-|W|(\delta-2)\le(n-\delta)(\delta+11\eta n-\delta+2)
  \leByRef{eq:ext-like:W} 11\eta n^2+2n\,.\] 
  In particular, by the definition of~$\xi$, by~\eqref{ext:chooseetad}
  and~\eqref{ext:chooseN}, 
  \begin{equation}\label{eq:ext:neighW}
    \Big| \big\{ v\in V(G)\setminus W \colon \deg(v,W)<|W|-\xi^2 n  \big\} \Big|
    \le \xi^2 n
    \,.
  \end{equation}
  In addition, by assumption~\ref{ext:assm6} we have $|B_i|\leq
  19m(2\delta-n)/(10n)$, which together with $\delta\le2n/3$,
  \eqref{ext:chooseetad}, \eqref{ext:chooseeps} and~\eqref{ext:smallxi} implies
  \begin{equation}\label{eq:ext:bigcupB}
    \Big|\bigcup B_i\Big|
    \le\frac{19}{10}(2\delta-n)
    \le\frac{19}{20}\delta
    <\delta-\xi n-(d+\eps)n
    \,.
  \end{equation}
  However, by assumption~\ref{ext:assm6} and the definition of an
  $(\eps,d)$-regular partition, vertices in $\bigcup B_i$ send at most
  $(d+\eps)n$ edges to $V(G)-\bigcup B_i - \bigcup I$. It follows from the
  definition of~$W$ that
  \begin{equation*}
    \bigcup B_i \cap W = \emptyset 
    \qquad
    \text{for all $i\in[k]$} \,.
  \end{equation*}
  Furthermore, \eqref{ext:chooseetad},
  \eqref{ext:chooseeps} and~\eqref{ext:smallxi} imply that
  $v\in\bigcup B_i$ has at least
  \begin{equation}
    \label{eq:ext-like:Bi}
    \delta-|W|-(d+\eps)n
    \geByRef{eq:ext-like:W} 2\delta-n-(d+\eps)n
    \gByRef{eq:ext:bigcupB} |\bigcup B_i|/2+32\xi^2 n
  \end{equation}
  neighbours in $\bigcup B_i$.  

  Now, for each $i\in[k]$ we let $A_i$ be the set of vertices in $\bigcup B_i$
  which are adjacent to at least $|W|-\xi^2 n$ vertices of $W$. In the rest of
  this paragraph we determine some important properties of the sets~$A_i$.
  By~\eqref{eq:ext:neighW} we have
  \begin{equation}\label{eq:ext:BminusA}
    \Big|\bigcup_{i\in[k]} \big( \bigcup B_i \big) \setminus A_i\big|\le\xi^2 n
    \qquad
    \text{for all $i\in[k]$}\,.
  \end{equation}
  But the vertices
  which are neither in~$W$ nor any of the sets~$A_i$ must be either in the
  exceptional set~$V_0$ or in $\bigcup B_i\setminus A_i$ for some $i$. 
  Hence we have
  \begin{equation}\label{eq:ext:V0Ai}
    \Big| V_0\cup \bigcup_{i\in[k]} \big( \bigcup B_i \big) \setminus A_i \Big|
    \le\varepsilon n+\xi^2 n<2\xi^2 n
    \,.
  \end{equation}
  Accordingly~\eqref{eq:ext-like:Bi} implies that 
  \begin{equation}\label{eq:ext:deltaAi}
    \delta(G[A_i])\geq |A_i|/2+30\xi^2 n
    \,,
  \end{equation}
  and since $|B_i|>\delta'-|I|\geq 2\delta'-m$ we have
  \begin{equation}
    \label{eq:ext-like:Ai}
    |A_i|\geq |\bigcup B_i|-2\xi^2 n\geq (1-\eps)\frac{n}{m}|B_i|-2\xi^2 n
    \ge 2\delta-n-3\xi^2 n
  \end{equation}
  for each $i\in[k]$, where we used the definition of~$\xi$,
  \eqref{ext:chooseetad}, \ref{ext:chooseeps}, and 
  \eqref{eq:ext:delta'} in the last inequality.

  \smallskip

  In the remainder of the proof we utilize the sets $A_i$ in order to find the
  desired squared path and squared cycles. We start by showing that we obtain
  squared cycles on $\ell$ vertices for each
  $\ell\in[3,\frac32|A_1|]\setminus\{5\}$. To see this note first that by
  Lemma~\ref{lem:nicepathcycle} (with $B=\emptyset$) we find in~$A_1$ a copy of
  $C_{2\ell'}$ for each $2\ell'\in \big[4,\min\{|A_1|,2\frac{n}{4}\}\big]$. By
  the definition of~$A_1$ every quadruple of consecutive vertices on such a
  cycle has at least $|W|-4\xi^2n$ common neighbours in~$W$, and by the
  definition of~$\xi$, \eqref{ext:chooseetad}, \eqref{ext:chooseeps}, and
  \eqref{eq:ext-like:W} we have $|W|-4\xi^2n\ge n/4$.  Hence we can apply
  Lemma~\ref{lem:squaringpath} to~$G$ and~$W$ to square this cycle. This gives
  us squared cycles of lengths $\ell$ with
  \begin{equation*}
    3\le\ell
    \le\min\Big\{\frac32|A_1|, 3\frac{n}{4}\Big\}
    \eqByRef{eq:ext:1120} \frac32|A_1|
  \end{equation*}
  such that~$\ell$ is divisible by three, but not of lengths not divisible by three.
 
  If we seek a squared cycle $C^2_{3\ell'+1}$ or $C^2_{3\ell'+2}$ (with
  $3\ell'+2\neq 5$) then we need to perform a process which we will call
  \emph{parity correction} and which we explain in the following two
  paragraphs. We shall use this parity correction process also in all remaining
  steps of the proof to obtain squared cycles of lengths not divisible by $3$.

  For obtaining a squared cycle of length $3\ell'+1$ we proceed as follows.  We
  pick (using Tur\'an's theorem) a triangle $abc$ in $A_1$ and \emph{clone} the
  vertex $b$, i.e., we insert a dummy vertex $b'$ into $G$ with the same
  adjacencies as $b$.  Then we apply Lemma~\ref{lem:nicepathcycle} to
  $A_1-\{b\}$ to find a path $P=(a,p_2,p_3,\ldots,p_{2\ell'-1},c)$ on $2\ell'$
  vertices whose end-vertices are $a$ and $c$. Finally we apply
  Lemma~\ref{lem:squaringpath} to the path $bPb'$, taking $Q_1=(b,a)$,
  $Q_2=(b,a,p_2,p_3)$ as the first quadruple and thereafter every other set of
  four consecutive vertices on $P$, finishing with
  $(p_{2\ell'-2},p_{2\ell'-1},c,b')$. This yields a squared path
  $(q_1,b,a,\ldots,c,b')$ on $3(\ell'+1)$ vertices, which gives a squared cycle
  $(b,a,\ldots,c)$ in~$G$ (without $q_1$ and the clone vertex~$b'$) on
  $3\ell'+1$ vertices as required.
 
  If we seek a squared cycle of length $3\ell'+2$ with $\ell'>1$ on the other
  hand, then we perform a similar process, except that we identify not one
  triangle in $A_1$ but two triangles $abc$, $xyz$ connected with an edge $cx$
  (which we obtain by the Erd\H{o}s-Stone theorem). We apply
  Lemma~\ref{lem:nicepathcycle} to find a path $P=(a,\ldots,z)$ in
  $A_1\setminus\{b,c,y,z\}$ on $2\ell'$ vertices. We then apply
  Lemma~\ref{lem:squaringpath} once to the path $bPy$ and once to $(b,c,x,y)$.
  Omitting the first vertex on each of the resulting squared paths and
  concatenating, we get a squared cycle $C^2_{3\ell'+2}$.


  Hence we do indeed obtain squared cycles $C^2_\ell$ for all
  $\ell\in[3,\frac32|A_1|]\setminus\{5\}$.  It remains to show that we can also
  find $C^2_\ell$ for all $\ell$ with $\frac32|A_1|\le\ell\le\sqc(n,\delta)$ and
  that we can find $P^2_{\sqp(n,\delta)}$.
  
  For this purpose, we first re-incorporate the vertices that are neither in $W$
  nor in any of the sets~$A_i$ by examining in which of the $A_i$ they have many
  neighbours. More precisely, for each $i\in[k]$, we let $X_i$ be~$A_i$ together
  with all vertices in $V(G)\setminus W$ which are adjacent to at least~$30\xi^2
  n$ vertices of $A_i$. Because every vertex in $V(G)\setminus W$ has at least
  $\delta-|W|$ neighbours outside $W$, by~\eqref{eq:ext-like:W} every vertex in
  $G-W$ is in $X_i$ for at least one $i$.  Moreover, by the definition of an
  $(\eps,d)$-regular partition, assumption~\ref{ext:assm6} and since $A_j\subset
  \bigcup B_j$, we have for all $j\in[k]$ with $j\neq i$ that
  \begin{equation}\label{eq:ext:AXempty}
    A_j\cap X_i = \emptyset
    \,.
  \end{equation}
  Hence it follows from~\eqref{eq:ext:V0Ai} that
  \begin{equation}\label{eq:ext:XA}
    |X_i|<|A_i|+2\xi^2 n 
    \qquad \text{and} \qquad
    |X_1-A_1|\le 2\xi^2 n
    \,.
  \end{equation}
  We finish the proof by distinguishing three cases.

  \smallskip

  {\sl Case 1:} $|X_i\cap X_j|\geq 2$ for some $i\neq j$.  Let $v_1$ and $v_2$
  be distinct vertices of $X_i\cap X_j$. Let $u_1$ and $u_2$ be distinct
  neighbours in $A_i$ of $v_1$ and $v_2$ respectively, and similarly $y_1$ and
  $y_2$ in $A_j$. Applying Lemma~\ref{lem:nicepathcycle} to $A_i$ we can find a
  path from $u_1$ to $u_2$ of length~$\ell'$ for any $4\le\ell'\le|A_i|-2$. We
  can find a similar path in $A_j$ from $y_1$ to $y_2$. Concatenating these
  paths with $v_1$ and $v_2$ we can find a $2\ell'$-vertex cycle $T_{2\ell'}$ in
  $X_1\cup X_2$ for any $10\leq 2\ell'\leq |A_i|+|A_j|-2$. There are no
  quadruples of consecutive vertices on $T_{2\ell'}$ using both $v_1$ and
  $v_2$. The four quadruples that use either~$v_1$ or~$v_2$ each have at least
  $(\xi-3\xi^2)n>100k$ common neighbours in $W$, where the inequality follows
  from~\eqref{ext:chooseN}, \eqref{ext:smallxi}, from
  \begin{equation}\label{eq:ext:knu}
    k\le\nu^{-1}\,,
  \end{equation}
  and from $\xi-3\xi^2>0$.  All other quadruples have at least $|W|-4\xi^2 n$
  common neighbours in $W$.  So applying Lemma~\ref{lem:squaringpath} we obtain
  a squared cycle on $3\ell'$ vertices. Again it is possible to perform parity
  corrections (prior to applying Lemma~\ref{lem:nicepathcycle}) so that in this
  case we have $C^2_\ell\subset G$ for every $\ell\in
  [3,\frac32(|A_i|+|A_j|-10)]\setminus\{5\}$.  By~\eqref{eq:ext-like:Ai}, we
  have $\sqc(n,\delta)\le\sqp(n,\delta)<\frac32(|A_i|+|A_j|-10)$.

  \smallskip

  {\sl Case 2:} for some $i$ every vertex of $A_i$ is adjacent to at least one
  vertex outside $X_i\cup W$. Since 
  \begin{equation*}
    |A_i|
    \geByRef{eq:ext:BminusA} \Big|\bigcup B_i \Big|-\xi^2n
    \geByRef{eq:BiBig} \frac43 \nu (1-\eps)n - \xi^2 n
    \geByRef{ext:smallxi} 13\xi n
    \gBy{\eqref{ext:smallxi},\eqref{eq:ext:knu}}
    31k\xi^2 n
  \end{equation*}
  we can certainly find $j\neq i$ such that there are $31\xi^2 n$ vertices in
  $A_i$ all adjacent to vertices of $X_j\setminus X_i$. Since no vertex of
  $X_j\setminus X_i$ is adjacent to $30\xi^2 n$ vertices of $A_i$ (by definition
  of $X_i$), we find two disjoint edges~$u_1v_1$ and~$u_2v_2$ from $u_1,u_2\in
  A_i$ to $v_1,v_2\in X_j$. Choosing distinct neighbours~$y_1$ of $v_1$ and $y_2$ of
  $v_2$ in $A_j$ and applying the same reasoning as in the previous case we are
  done.

  \smallskip
 
  {\sl Case 3:} for each $i\neq j$ we have $|X_i\cap X_j|\leq 1$, and for each
  $i$ some vertex in $A_i$ is adjacent only to vertices in $W\cup X_i$. Thus
  $|X_i|\geq \delta-|W|+1$ for each $i$. We now first focus on finding a squared
  path on $\sqp(n,\delta)$ vertices in $G$, and then turn to the squared cycles
  which will complete the proof.  If for some $i\neq j$ we have $|X_i\cap
  X_j|=1$ then we obtain a squared path of the desired length as in Case~1.
  There we required two vertices in $X_i\cap X_j$ to obtain a squared cycle
  (which must return to its start), but one vertex suffices for a squared path
  to cross from $X_i$ to $X_j$.

  So, assume
  that the sets $X_i$ are all disjoint.  It follows that
  $k\leq(n-|W|)/(\delta-|W|+1)$.  Since $|W|\leq n-\delta$
  by~\eqref{eq:ext-like:W}, this implies
  \[
  k\le\frac{n-(n-\delta)}{\delta-(n-\delta)+1}
  =\frac{\delta}{2\delta-n+1}\,.
  \]
  Now if $k\geq\rp(n,\delta)+1$ then we would have
  \[ \rp(n,\delta)+1\leq k\leq \frac{\delta}{2\delta-n+1}\,,\]
  and so
  \[ \rp(n,\delta)\leq \frac{n-\delta-1}{2\delta-n+1}\,,\] but by~\eqref{eq:r}
  we have $\rp(n,\delta)\ge\frac{n-\delta}{2\delta-n+1}$, so
  \begin{equation*}
    k\leq\rp(n,\delta) \,. 
  \end{equation*}
  Thus the largest of the sets $X_i$, say $X_1$, has at
  least
  \begin{equation}\label{eq:ext:X1}
    |X_1|\ge \frac{n-|W|}{k} 
    \geByRef{eq:ext-like:W} \frac{\delta}{k}
    \ge\frac{\delta}{\rp(n,\delta)}
  \end{equation}
  vertices. 

  We now want to apply Lemma~\ref{lem:nicepathcycle} with $H=G[X_1]$ and `bad'
  vertices $B=X_1-A_1$. Note that by~\eqref{eq:ext:XA} there are at most $2\xi^2
  n$ vertices in $B=X_1-A_1$, and so we have
  \begin{equation*}
    |B|\leByRef{eq:ext:XA} 2\xi^2 n
    \leByRef{ext:smallxi} \frac{\nu\delta}{100}
    \leByRef{eq:ext:knu} \frac{\delta}{100k}\le\frac{|H|}{100} \,.
  \end{equation*}
  Moreover, $\delta(H)=\delta(G[X_1])\geq 30\xi^2 n$ by definition of~$X_1$, and therefore
  every vertex in~$B$ has at least $30\xi^2 n\ge 9\cdot 2 \xi^2 n\ge 9|B|$
  neighbours in~$H$.
  In addition, vertices~$v$ in $H-B\subset A_1$ satisfy
  \begin{equation*}\begin{split}
    \deg(v,X_1)
    & \geByRef{eq:ext:deltaAi} \frac{|A_1|}{2}+30\xi^2 n
    \gByRef{eq:ext:XA} \frac{|X_1|}{2}+25\xi^2 n \\
    & = \frac{|H|}{2}+25\xi^2 n
    \geByRef{ext:chooseN} \frac{|H|}{2}+9|B|+10 \,.
  \end{split}\end{equation*}
  Hence we can indeed apply Lemma~\ref{lem:nicepathcycle}, to obtain a path $T$
  covering $\min\{X_1,n/2\}$ vertices on which every quadruple of consecutive
  vertices contains at most one `bad' vertex. Finally we want to apply
  Lemma~\ref{lem:squaringpath} to the graph~$G[X_1\cup W]$ and the
  cycle~$T$ with the following ordering~$\sigma$ of the quadruples of
  consecutive vertices in~$T$: $\sigma$ is such that
  all quadruples containing vertices from $B$ come first, followed (by an
  arbitrary ordering of) all other quadruples. There are at most $2\cdot 2\xi^2
  n$ quadruples containing vertices from~$B=X_1-A_1$, and by the definition
  of~$A_1$ and of~$W$, each of them has at least
  $(\xi-3\xi^2)n\ge\xi^2 n$ common neighbours in~$W$. All remaining quadruples
  have, by the definition of~$A_1$, by~\eqref{eq:ext-like:W} and since
  $\delta\le\ 2n/3$, at least $|W|-4\xi^2
  n\ge\tfrac{n}{4}\ge\frac12\min\{|X_1|,\frac{n}{2}\}$ common neighbours
  in~$W$. Hence, we can indeed apply Lemma~\ref{lem:squaringpath} to obtain a
  squared path on at least $\frac32\min\{|X_1|,n/2\}\ge\sqp(n,\delta)$
  vertices, where the inequality follows from the definition of
  $\sqp(n,\delta)$, from~\eqref{eq:ext:1120}, and from~\eqref{eq:ext:X1}.
 
  At last, we show that we can find in $G$ the desired long squared cycles in
  Case~3.
  Assume first that there is a cycle of sets (relabelling the indices if
  necessary) $X_1,X_2,\ldots,X_s$ for some $3\leq s\leq k$ such that $X_i\cap
  X_{i+1\!\mod s}=\{v_i\}$ for each $i$, and the $v_i$ are all distinct, then for
  each $i$ we may choose neighbours $u_i\in A_i$ and $y_i$ in $A_{i+1\!\mod s}$ of
  $v_i$, and we may insist that all these $3s$ vertices are distinct. Similarly
  as before we can apply Lemma~\ref{lem:nicepathcycle} to each $G[A_i]$ in turn
  and concatenate the resulting paths, in order to find a cycle $T_{2\ell'}$ for
  every $8s\leq 2\ell'\leq |A_1|+|A_2|$ on which there are no quadruples using
  more than one vertex outside $\bigcup_i A_i$. Again (checking the conditions
  similarly as before) we may apply Lemma~\ref{lem:squaringpath} to $T_{2\ell'}$
  to obtain a squared cycle on $3\ell$ vertices.  Finally by performing parity
  corrections we obtain $C^2_\ell$ for every
  $\ell\in[3,\frac32(|A_1|+|A_2|)]\setminus\{5\}$.

  If there exists no such cycle of sets, then $\sum_{i=1}^k |X_i|\leq
  n-|W|+k-1$. Since we have also $|X_i|\geq \delta-|W|+1$ for each $i$ and
  $|W|\leq n-\delta$, it follows from the definition of $\rc(n,\delta)$ (by
  establishing a relation similar to~\eqref{eq:r}) that $k\leq \rc(n,\delta)$,
  and by averaging, that the largest of the sets~$X_i$, say $X_1$, contains at
  least $2\sqc(n,\delta)/3$ vertices. As before, we can apply
  Lemma~\ref{lem:nicepathcycle} to $X_1$ to obtain a cycle $T_{2\ell'}$ for each
  $4\leq 2\ell'\leq |X_1|$ on which the `bad' vertices from~$B=X_1-A_1$ are
  separated, and apply Lemma~\ref{lem:squaringpath} to it to obtain a squared
  cycle $C^2_{3\ell'}$ for each $6\leq 3\ell'\leq\sqc(n,\delta)$ as
  required. Again the parity correction procedure is applicable, so we get
  $C^2_\ell$ for every $\ell\in[3,\sqc(n,\delta)]\setminus\{5\}$.
\end{proof} 
  
\section{Concluding remarks} 
 
\paragraph{\bf The proof of Theorem~\ref{thm:main}.} 

Our results were most difficult to prove for $\delta\approx 4n/7$. This is
somewhat surprising given the experience from the partial and perfect
packing results of Koml\'os~\cite{Kom} and K\"uhn and
Osthus~\cite{KuhOst}. In the setting of these results it becomes steadily
more difficult to prove packing results as the minimum degree of the graph
(and hence the required size of a packing) increases, with perfect packings
as the most difficult case. Yet in our setting it is relatively easy to
prove our results when the minimum degree condition is large. This
difference occurs because we have to embed triangle-connected graphs; as
the minimum degree increases the possibilities for bad behaviour when
forming triangle-connections are reduced. 

This is paralleled by the behaviour of $K_4$-free graphs: For any
minimum degree $\delta(G)>2v(G)/3$ the graph $G$ is not $K_4$-free.
However, if $\delta(G)>5v(G)/8$ then by the Andr\'asfai-Erd\H os-S\'os
theorem~\cite{AES} the $K_4$-free graph $G$ is forced to be tripartite, while
for smaller values of~$\delta(G)$ there exist more possibilities.

\smallskip

\paragraph{\bf Extremal graphs.}

It is straightforward to check (from our proofs) that up to some trivial
modifications the graphs $G_p(n,\delta)$ and $G_c(n,\delta)$ are the only extremal
graphs. 
We believe that the graph $G_p(n,\delta)$ remains extremal for squared
paths even when $\delta$ is not bounded away from $n/2$, although as noted
in Section~\ref{sec:Intro} the same is not true for~$G_c(n,\delta)$ and
squared cycles.

However it is not the case that the only extremal graph excluding some
$C^2_\ell$ of chromatic number four is $K_{n-\delta,n-\delta,2\delta-n}$
(cf.\ \ref{thm:main:ii} of our main theorem, Theorem~\ref{thm:main}). Let
us briefly explain this.  Suppose $\ell$ is not divisible by three. Since
$C^2_\ell$ has no independent set on more than~$\lfloor\ell/3\rfloor$
vertices, whenever we remove an independent set from $C^2_\ell$ we must
leave some three consecutive vertices, which form a triangle. Now suppose
that we can find a graph $H$ on $\delta$ vertices with minimum degree
$2\delta-n$ which is both triangle-free and contains no even cycle on more
than $2(2\delta-n)$ vertices. Then the graph $G$ obtained by adding an
independent set of size $n-\delta$ to $H$, all of whose vertices are
adjacent to all of $H$, contains no squared cycle of length indivisible by
three and no squared cycle with more than $3(2\delta-n)$ vertices.

To mention one possible $H$, take $\delta=\frac{6n}{11}$ and let $H$ be
obtained as follows. We take the disjoint union of three copies of
$K_{n/11,n/11}$ and fix a bipartition. Now we add three vertex disjoint
edges within one of the two partition classes, one between each copy of
$K_{n/11,n/11}$. The resulting triangle-free graph has no even cycle
leaving any copy of $K_{n/11,n/11}$.  Hence all even cycles have at most
$\frac{2n}{11}$ vertices. However, it has odd cycles of up to
$\frac{6n}{11}-3$ vertices.

\smallskip

\paragraph{\bf Long squared cycles.} 

Theorem~\ref{pcycle}~\ref{pcycle:missing} states that if any of various odd
cycles are excluded from $G$ we are guaranteed even cycles of every length
up to $2\delta(G)$, whereas the equivalent statement in our
Theorem~\ref{thm:main} contains an error term. We believe this error term
can be removed, but at the cost of significantly more technical work,
requiring both a new version of the stability lemma and new extremal
results.

\smallskip
 
\paragraph{\bf Higher powers of paths and cycles.}

We note that Theorem~\ref{thm:KSSPosa} has a natural generalisation to
higher powers of cycles, the so called P\'osa-Seymour Conjecture. This
conjecture was proved for all sufficiently large $n$ by Koml\'os,
S\'ark\"ozy and Szemer\'edi~\cite{KSSgen}. We conjecture a natural
generalisation of Theorem~\ref{thm:main} for higher powers of paths and
cycles.
 
Given $k$, $n$ and $\delta$, we construct an $n$-vertex graph
$G_p^{(k)}(n,\delta)$ by partitioning the vertices into an `interior' set
of $\ell=(k-1)(n-\delta)$ vertices upon which we place a complete balanced
$k-1$-partite graph, and an `exterior' set of $n-\ell$ vertices upon which
we place a disjoint union of $\lfloor(n-\ell)/(\delta-\ell+1)\rfloor$
almost-equal cliques. We then join every `interior' vertex to every
`exterior' vertex. We construct $G_c^{(k)}(n,\delta)$ similarly, permitting
the cliques in the `exterior' vertices to overlap in cut-vertices of the
`exterior' set if this reduces the size of the largest clique while
preserving the minimum degree $\delta$.
 
\begin{conjecture} 
  Given $\nu>0$ and $k$ there exists $n_0$ such that whenever $n\geq n_0$
  and $G$ is an $n$-vertex graph with $\delta(G)=\delta>\frac{k-1}{k}n+\nu
  n$, the following hold.
  \begin{enumerate}[label=\irom]
  \item If $P^k_\ell\subset G_p^{(k)}(n,\delta)$ then $P^k_\ell\subset G$. 
  \item If $C^k_{(k+1)\ell}\subset G_c^{(k)}(n,\delta)$ for some integer
     $\ell$, then $C^k_{(k+1)\ell}\subset G$.
   \item If $C^k_\ell\subset G_c^{(k)}(n,\delta)$ with $\chi(C^k_{\ell})=k+2$
     and $C^k_\ell\not\subset G$ for some integer $\ell$, then
     $C^k_{(k+1)\ell}\subset G$ for each integer $\ell<k\delta-(k-1)n-\nu n$.
  \end{enumerate}
\end{conjecture} 
 
It seems likely that again the $\nu n$ error term in the last statement is
not required, but again (at least for powers of cycles) it is required in
the minimum degree condition.

\section*{Acknowledgement}

This project was started at DocCourse 2008, organised by the 
research training group Methods for Discrete Structures, 
Berlin. In particular, we would like to thank Mihyun Kang and 
Mathias Schacht for organising this nice event.

 
\bibliographystyle{amsplain} 
\bibliography{TuranPosa} 

 
\appendix 

\section{Proof of Lemma~\ref{lem:bl}} 
 
For the proof of Lemma~\ref{lem:bl} we apply the following version (which
is a special case) of the Blow-up Lemma of Koml\'os, S\'ark\"ozy and
Szemer\'edi~\cite{KSS_bl}.
 
\begin{lemma}[Blow-up Lemma~\cite{KSS_bl}]
\label{lem:KSSbl} 
  Given fixed $c,d>0$, there exist $\eps_0>0$ and $n_\subsc{BL}$ such that
  for any $0<\eps<\eps_0$ the following holds. Let~$H$ be any graph on at
  least $n_\subsc{BL}$ vertices with $V(H)=V_1\dcup V_2\dcup V_3$ and
  $|V_i|\ge\frac16|V(H)|$, in which each bipartite graph $H[V_i,V_j]$ is
  $(3\eps,d)$-regular and furthermore $\delta_{V_i}(V_j)\geq \frac12d|V_i|$
  for each $1\leq i,j\leq 3$.

  Let $F$ be any subgraph of the complete tripartite graph with parts
  $V_1$,~$V_2$ and $V_3$ such that the maximum degree of $F$ is at most
  four.  Assume further, that at most four vertices $x_i$ ($i\in[4]$) of $F$ are
  endowed with sets $C_{x_i}\subset V_j$ such that $x_i\in V_j$ and
  $|C_{x_i}|\ge c|V_j|$

  Then there is an embedding $\psi\colon V(F)\rightarrow V(H)$ of $F$ into $H$
  with $\psi(x_i)\in C_{x_i}$ for $i\in[4]$.
 \end{lemma} 
 
We also say that the~$x_i$ in Lemma~\ref{lem:KSSbl} are \emph{image restricted}
to~$C_{x_i}$.
 
\begin{proof}[Proof of Lemma~\ref{lem:bl}] Given $d$, we let $c=d^2/4$. Now
  Lemma~\ref{lem:KSSbl} gives values $\eps_0>0$ and $n_\subsc{BL}$. We choose
  $\eps_\subsc{EL}=\min(\eps_0,d^2/24)$. Given $\eps<\eps_\subsc{EL}$ and
  $m_\subsc{EL}$ we choose
  \[n_\subsc{EL}=\max\left(2m_\subsc{EL}n_\subsc{BL},\frac{6m^4}{\eps}\right)\,.\]
  Let $n\ge n_\subsc{EL}$, let $G$ be an $n$-vertex graph, and let $R'$ be an
  $(\varepsilon,d)$-reduced graph of $G$ on $m\le m_\subsc{EL}$ vertices.
 
  Fix a set $\mathcal{T'}=\{T'_1,\ldots,T'_{\CTF(R')/3}\}$ of vertex-disjoint
  triangles in a triangle component of $R'$ covering $\CTF(R')$ vertices. For
  each triangle $T'_i=X'_{i,1}X'_{i,2}X'_{i,3}$ we may by regularity for each
  $j\in[3]$ remove at most $\eps|X'_{i,j}|$ vertices from $X'_{i,j}$ to obtain a
  set $X_{i,j}$ such that each pair $(X_{i,j},X_{i,k})$ is not only
  $(2\eps,d)$-regular but also satisfies $\delta_{X_{i,k}}(X_{i,j})\geq
  (d-3\eps)|X_{i,k}|$. We let $R$ be the $(2\eps,d)$-reduced graph corresponding
  to the new vertex partition given by replacing each $X'_{i,j}$ with $X_{i,j}$;
  then every edge of $R'$ carries over to $R$, and we let $\mathcal{T}$ be the
  set of $\CTF(R')/3$ vertex disjoint triangles in $R$ corresponding
  to~$\mathcal{T'}$. We set $r=\CTF(R')/3$.

  Our strategy now is as follows. We shall first fix a collection of
  suitable triangle walks $W_1,\dots, W_{r-1}$ and~$W'$ in~$R$.
  Next, for each of these triangle walks $W=(E_1,E_2,\ldots)$ we do the
  following. Let $\overrightharp{U_1V_1}$ be (a suitable) orientation of
  the first edge $E_1$ of~$W$. We shall construct a
  sequence~$Q(W,\overrightharp{U_1V_1})$ of vertices of~$R$ whose first two
  vertices are~$U_1$ and~$V_1$, in that order, and which has the property
  that every vertex in the sequence is adjacent to the two preceding
  vertices (as is the case for a squared path). Then we use this
  sequence~$Q(W,\overrightharp{U_1V_1})$ to obtain a squared path in~$G$
  following $W$, whose first two vertices are in $U_1$ and
  $V_1$. Finally, connecting suitable paths appropriately will lead to a
  proof of \ref{lem:bl:1}, \ref{lem:bl:2}, and \ref{lem:bl:3}.

  We first construct the triangle walks $W_1,\dots, W_{r-1}$
  and~$W'$.  For each $1\leq i\leq r-1$ let $W_i$ be a fixed
  triangle walk in $R$ whose first edge is in $T_i$ and whose last is in
  $T_{i+1}$. We suppose (repeating edges in the triangle walk $W_i$ if
  necessary) that each triangle walk~$W_i$ contains at least ten edges,
  that the first edge of $W_{i+1}$ is not the same as the last edge of
  $W_i$, and such that each walk with more than ten edges is of minimal
  length. We have $|W_i|\leq\binom{m}{2}$ for each $i$. Let $W'$ be the
  triangle walk obtained by concatenating $W_1,\ldots,W_{r-1}$. 

  Next, we describe how to construct the sequence $Q(W,\overrightharp{A_1B_1})$
  for any triangle walk $W=(E_1,E_2,\ldots)$ in $R$ and any orientation
  $\overrightharp{A_1B_1}$ of its first edge $E_1$.  We construct
  $Q(W,\overrightharp{A_1B_1})$ iteratively as follows. Let $Q_1=(A_1,B_1)$. Now
  for each $2\leq i\leq |W|$ successively, we define $Q_i$ as follows. The last
  two vertices $A_{i-1},B_{i-1}$ of $Q_{i-1}$ are an orientation of~$E_{i-1}$.
  If $E_i=A_{i-1}B_i$ we create~$Q_i$ by appending $(B_i,A_{i-1})$ to~$Q_{i-1}$;
  if $E_i=B_{i-1}B_i$ we append $(B_i,A_{i-1},B_{i-1},B_i)$ to~$Q_{i-1}$ to
  create~$Q_i$. At each step the final two vertices of $Q_i$ are an orientation
  of~$E_i$. Furthermore every vertex of $Q_i$ is adjacent in $R$ to the two
  vertices preceding it in $Q_i$. Finally, we let
  $Q(W,\overrightharp{A_1B_1})=Q_{|W|}$.

  We shall need the following observations concerning the lengths of sequences
  constructed in this way. It is easy to check by induction that for any
  triangle walk $W$ with at least two edges whose first edge is~$U_1V_1$, we
  have
  \begin{equation}\label{eq:1mod3}
    |Q(W,\overrightharp{U_1V_1})|+|Q(W,\overrightharp{V_1U_1})|\equiv 1 \mod 3\,.
  \end{equation}
  Now consider the concatenation~$W'$ of the walks~$W_i$.  Let~$U_1V_1$ be
  the first edge of~$W_1$. If we construct $Q(W', \overrightharp{U_1V_1})$
  then the first edge~$U_iV_i$ and the last edge $U'_iV'_i$ of each~$W_i$
  obtains an orientation, say $\overrightharp{U_iV_i}$
  and~$\overrightharp{U'_iV'_i}$. Clearly, there are sequences~$\tilde Q_i$ of
  vertices in~$T_i$ for $1<i<r$, such that $Q(W',\overrightharp{V_1U_1})$ is the concatenation of
  \[
  Q(W_1,\overrightharp{V_1U_1}),\tilde
  Q_2,Q(W_2,\overrightharp{V_2U_2}),\dots,\tilde Q_{r-1},
  Q(W_{r-1},\overrightharp{V_{r-1} U_{r-1}})\,.\]
  Further we let~$\tilde Q_1=T_1-U_1V_1$ and $\tilde Q_r=T_r-U'_{r-1}V'_{r-1}$.
  We define $f_i=|\tilde Q_i|\!\mod 3$ for~$i\in[r]$.
  Together with~\eqref{eq:1mod3} we obtain
  \begin{equation*}
    |Q(W',\overrightharp{U_1V_1})|+ |Q(W_1,\overrightharp{V_1U_1})|+ 
    \sum_{1<i<r} \big( |Q(W_i,\overrightharp{V_iU_i})|+ f_i \big)
    \equiv 1 \mod 3 
  \end{equation*}
  and hence
 \begin{equation}\label{eq:0mod3}
    |Q(W',\overrightharp{U_1V_1})|+ 
    \sum_{i\in[r-1]} \big( |Q(W_i,\overrightharp{V_iU_i})|+f_i \big)
    + f_r
    \equiv 0 \mod 3 \,.
 \end{equation}
 This will enable us to construct cycles of lengths divisible by three later.

  In order to construct squared paths in~$G$ from short vertex sequences in~$R$
  we use the following fact.

  \begin{fact}\label{fac:pathlink} 
    Let $X_1,X_2,X_3$ be vertices of $R$ (not necessarily distinct), and~$Z$ be
    any set of at most $2\eps |X_1|$ vertices of $G$. Suppose that $X_1X_2$ and
    $X_1X_3$ are edges of $R$. Suppose furthermore that we have two vertices
    $u$ and $v$ of $G$ such that $u$ and $v$ have at least $(d-2\eps)^2|X_1|$
    common neighbours in $X_1$, and $v$ has at least $(d-2\eps)|X_2|$
    neighbours in $X_2$.
 
    Then there is a vertex $w\in X_1-Z$ adjacent to $u$ and $v$ such that
    $v$ and $w$ have at least $(d-2\eps)^2|X_2|$ common neighbours in $X_2$
    and $w$ has at least $(d-2\eps)|X_3|$ neighbours in $X_3$.
  \end{fact} 
  \begin{factproof}[Proof of Fact~\ref{fac:pathlink}] 
    Let $W$ be the set of common neighbours of $u$ and $v$ in $X_1$. Since
    $X_1X_2\in E(R)$, at most $2\eps |X_1|$ vertices of $W$ have fewer than
    $(d-2\eps)|\Gamma(v)\cap X_2|\geq (d-2\eps)^2|X_2|$ common neighbours
    with $v$ in~$X_2$. Since $X_1X_3\in E(R)$ at most $2\eps |X_1|$
    vertices of $W$ have fewer than $(d-2\eps)$ neighbours in
    $X_3$. Finally since $6\eps |X_1|<(d-2\eps)^2|X_1|$ we can find a
    vertex of $W\setminus Z$ satisfying the desired properties.
  \end{factproof} 

  With these buiding bricks at hand we can now turn to the proofs of
  \ref{lem:bl:1}, \ref{lem:bl:2}, and \ref{lem:bl:3}.

  \smallskip

  {\sl Proof of~\ref{lem:bl:1}}, i.e.,
  $G$ contains $C^2_{3\ell}$ for each $3\ell\leq
  (1-d)\CTF(R)n/m$:
  When $\ell\leq (1-d)n/m$ we have $C^2_{3\ell}\subset
  K_{(1-d)n/m,(1-d)n/m,(1-d)n/m}$ and thus by Lemma~\ref{lem:KSSbl} we can
  find $C^2_{3\ell}$ as a subgraph of $G$ (whose vertices are in $T_1$,
  with no restrictions required).
  Otherwise we use the following strategy.
  Let $UV$ be the first edge of the triangle walk $W_1$.
 
  Our first goal will be to construct a squared path~$P'$ in~$G$ which
  `connects' $T_1$ to~$T_2$, $T_2$ to~$T_3$, and so on. For this purpose we
  choose two adjacent vertices $u$ and $v$ of $G$ in $U$ and $V$
  respectively, such that $u$ and $v$ have $(d-2\eps)^2n/m$ common
  neighbours in both the third vertex of $T_1$ and the third vertex of
  $Q(W',\overrightharp{UV})$, such that $v$ has $(d-2\eps)n/m$ neighbours
  in the fourth vertex of $Q(W',\overrightharp{UV})$, and such that $u$ has
  $(d-2\eps)n/m$ neighbours in $V$. This is possible by the regularity of
  the various pairs.  (Observe that the required sizes for the
  neighbourhoods and joint neighbourhoods are chosen large enough for an
  application of Lemma~\ref{lem:KSSbl} in the triangle~$T_1$.)  Now we
  apply Fact~\ref{fac:pathlink} with the vertices~$u$ and~$v$ and the
  third, fourth and fifth vertices of $Q(W',\overrightharp{UV})$ to obtain
  a third vertex $v'$ in the third vertex of $Q(W',\overrightharp{UV})$
  such that~$u$ and~$v$ are adjacent to~$v'$.  By repeatedly applying
  Fact~\ref{fac:pathlink} we construct a sequence of vertices $P'$
  (starting with $u,v,v'$), where the $i$th vertex of $P'$ is in the $i$th
  set of $Q(W',\overrightharp{UV})$ and is adjacent to its two
  predecessors, and where the vertices are all distinct (noting that
  $3|W'|<\eps n/m$). Thus~$P'$ is a squared path running from $T_1$ to
  $T_{r-1}$ following all the triangle walks~$W_i$.
 
 In addition we construct similarly (and without re-using vertices) for each $1\leq
  i\leq r-1$ a squared path $P_i$ following the triangle walk
  $W_i$. However, this time we use the opposite orientation for the first edge: that
  is, instead of constructing $P_1$ from $Q(W_1,\overrightharp{UV})$ we use
  $Q(W_1,\overrightharp{VU})$, and similarly for each $P_i$ we use the
  opposite orientation of the first edge of $W_i$ to that used in
  $P'$. Again, for each $P_i$ we insist that the first two vertices have
  suitable neighbourhoods in $T_i$, and the last two in $T_{i+1}$, for
  an application of Lemma~\ref{lem:KSSbl} in these triangles. Again, this is
  possible by regularity. 

  We note that the total number of vertices on all of these squared paths
  is not more than $6m\binom{m}{2}<\eps n/m$. Finally, we remove from $T_1$
  all vertices of $P=P'\cup P_1\cup\cdots\cup P_{r-1}$. Since at
  most $\eps n/m$ vertices are removed, and each cluster of $T_1$ has size at
  least $(1-3\eps)n/m$, even after removal all three pairs remain
  $(3\eps,d)$-regular and each cluster still has size at least
  $(1-4\eps)n/m$.

  Thus the conditions of Lemma~\ref{lem:KSSbl} are satisfied, and hence we
  may embed a squared path $S_1$ into $T_1$, with the four restrictions
  that its first vertex is a common neighbour of the first two vertices of
  $P'$, its second a neighbour of the first vertex of $P'$, its penultimate
  vertex a neighbour of the first vertex of $P_1$ and its final vertex a
  common neighbour of the first two vertices of $P_1$ (noting that by
  choice of the first two vertices of~$P'$ and of~$P_1$ the sets to which
  these vertices are restricted are indeed of size $cn/m$ because $c=
  d^2/4$).  In this way we can construct a squared path on $3\ell_1+f_1$
  vertices for any integer $\ell_1\in [10,(1-d)n/m]$ (since $3\cdot 4\eps<d$),
  where $f_1\in\{0,1,2\}$ is as defined above~\eqref{eq:0mod3}.
  Similarly we may apply Lemma~\ref{lem:KSSbl} to each
  $T_i$ ($2\leq i\leq r$), after removing~$P$ from~$T_i$, to obtain
  squared paths $S_i$ of length $3\ell_i+f_i$ for any integer $\ell_i\in
  [10,(1-d)n/m]$.

  Finally $S=P'\cup S_1\cup P_1\cup\ldots\cup P_{r-1}\cup
  S_{r}$ forms a squared cycle in~$G$.  It follows
  from~\eqref{eq:0mod3} that the length of~$S$ is divisible by three.
  We conclude that indeed $S=C^2_{3k}$, where we may choose any integer~$k$
  with $3k\in [6m^3,(1-d)\CTF(R)n/m]$, as required.
 
  \smallskip

  {\sl Proof of~\ref{lem:bl:2}:\ }
  When every triangle component of $R$ contains $K_4$ we also want to obtain
  squared cycles whose lengths are not divisible by three. Observe that if
  $ABCD$ is a copy of $K_4$ in $R$, then the vertex sequences $ABC$,
  $ABCDABC$ and $ABCDABCDABC$ each start and end with the same pair. Hence,
  with the help of Fact~\ref{fac:pathlink}, these sequences can be used to
  construct squared paths in~$G$ of length~$3$ (which is $0\!\mod 3$),
  length~$7$ ($1\!\!\mod 3$), and length $11$ ($2\!\mod 3$).

  We construct $C^2_\ell$ for $\ell\in[3,20]\setminus\{5\}$ within a copy
  of $K_4$ in $R$ directly (by the above methods). To obtain $C^2_\ell$
  with $21\leq\ell\leq 3(1-d)n/m$ we remove at most $2\eps n/m$ vertices
  from each of $A$, $B$ and $C$ to obtain a triangle satisfying the
  conditions of Lemma~\ref{lem:KSSbl}, construct a short path in $A,B,C,D$
  following the appropriate vertex sequence for $\ell\!\mod 3$ and apply
  Lemma~\ref{lem:KSSbl} to obtain $C^2_\ell$. Finally, to obtain longer
  squared cycles we perform the same construction as above, with the
  exception that $W'$ is any triangle walk to and from a copy of $K_4$, and
  so $Q(W',\overrightharp{UV})$ may be taken (using one of the three vertex
  sequences above) to have any desired number of vertices modulo three (and
  not more than $2m^2$ in total).

  \smallskip

  {\sl Proof of~\ref{lem:bl:3}:\ } 
  Lastly, when we are required to construct a squared path between two
  specified edges $u_1v_1$ (with $2dn/m$ common neighbours in both $X_1$
  and $Y_1$) and $u_2v_2$ (with $2dn/m$ common neighbours in both $X_2$ and
  $Y_2$) using triangles $T$ in $R$, we apply the identical strategy,
  noting that the conditions on $u_1v_1$ and $u_2v_2$ are already suitable
  for an application of Fact~\ref{fac:pathlink}.
\end{proof} 
\end{document}